\documentclass[11pt]{amsart}
\usepackage{amssymb,amsmath}

\makeatletter
\theoremstyle{plain}
 \newtheorem{thm}{Theorem}[section]
\newtheorem{thm*}{Theorem}
 \newtheorem{lem}[thm]{Lemma}
 \newtheorem{prop}[thm]{Proposition}
 \newtheorem{cor}[thm]{Corollary}
 \numberwithin{equation}{section} 
\numberwithin{figure}{section} 
 \theoremstyle{plain}
 \theoremstyle{definition}
 \newtheorem{defn}[thm]{Definition}
 \newtheorem{rem}[thm]{Remark}

\newcommand{\calS}{{\mathcal{S}}}

\newcommand{\calL}{{{\mathcal L}}}

\newcommand{\C}{{{\mathbb C}}}

\newcommand{\R}{{{\mathbb R}}}

\newcommand{\bH}{{{\bf H}}}

\newcommand{\cH}{{{\mathcal H}}}

\newcommand{\bX}{{{\bf X}}}
\newcommand{\bY}{{{\bf Y}}}

\newcommand{\bZ}{{{\bf Z}}}

\newcommand{\bT}{{{\bf T}}}
\newcommand{\cC}{{{\mathcal C}}}

\newcommand{\V}{{{\mathbb V}}}

\newcommand{\Z}{{{\mathbb Z}}}
\newcommand{\fH}{{{\mathfrak H}}}

\newcommand{\cV}{{{\mathcal V}}}

\makeatother

\begin{document}

\title[Quasiconformal mappings in $\fH^\star$]{Quasiconformal mappings in the hyperbolic Heisenberg group and a lifting theorem}

\author[I.D. Platis]{Ioannis D. Platis}

\email{jplatis@math.uoc.gr}
\address{Department of Mathematics and Applied Mathematics\\
 University of Crete
\\ University Campus\\
GR 700 13 Voutes Heraklion Crete\\Greece}

\begin{abstract}
A study of smooth contact quasiconformal mappings of the hyperbolic Heisenberg group is presented in this paper. Our main result is a Lifting Theorem; according to this, a symplectic quasiconformal mapping of the hyperbolic plane can be lifted to a circles preserving quasiconformal mapping of the hyperbolic Heisenberg group.
\end{abstract}

\keywords{Quasiconformal mappings, Heisenberg group, hyperbolic Heisenberg group, lifting theorems.\\
\;\;{\it 2010 Mathematics Subject Classification:} 30L10, 30C60, 30C75.}

\thanks{{\it Acknowledgements.} A big part of the material presented in this article lies in G. Moulantzikos' M.Sc. Thesis \cite{Mou}, which he carried out under my supervision. I take here the opportunity to thank George for his patience and endurance in clarifying many subtle points of the article in hand.}
\maketitle

\tableofcontents

\section{Introduction}
The Heisenberg group $\fH$ is the Lie group $(\C\times\R,*)$ where the group law is given by
$$
(z,t)*(z',t')=\left(z+z',t+t'+2\Im(\overline{z'}z)\right).
$$
$\fH$ has several rich structures and is being well studied continuously for many years. In the first place, $\fH$ is a 2-step nilpotent Lie group and it constitutes a primary example of sub-Riemannian geometry. The horizontal space of that geometry is generated by left invariant vector fields which lie in the kernel of the 1-form
$$
\omega=dt+2\Im(\overline{z}dz).
$$
The form $\omega$ is contact: $\omega\wedge d\omega$ is a constant multiple of the Euclidean volume form of $\C\times\R$ and therefore $\fH$ is a contact manifold. The Heisenberg group $\fH$ also arises naturally in the context of complex hyperbolic geometry:
It is well known that the one point compactification of the boundary $\partial\bH^2_\C$ of complex hyperbolic plane $\bH^2_\C=\{(z_1,z_2)\in\C^2\;|\;2\Re(z_1)+$ $|z_2|^2<0\}$ may be identified to set $\fH\cup\{\infty\}$. The boundary $\partial\bH^2_\C$ is a strictly pseudoconvex domain in $\C^2$ with contact form $\omega'$. If $\Psi$ is the map from $\fH$ to $\partial\bH^2_\C$, then $\Psi\omega'=\omega$. 

As a metric space, $\fH$ may be endowed with two equivalent metrics. The first one is the Carnot-Carath\'eodory (CC) path metric $d_{cc}$ which arises from the sub-Riemannian structure of $\fH$ and corresponds to the sub-Riemannian symmetric tensor $g_{cc}$ defined on the horizontal space. The second is the Kor\'anyi-Cygan metric $d_{\fH}$ which is not a path metric. 
Considering $\fH$ as a metric space (with respect to any of the two $d_{cc}$ and $d_{kr}$), Kor\'anyi and Reimann developed the celebrated theory of quasiconformal mappings of $\fH$, see for instance the fundamental articles \cite{KR1, KR2}. This theory followed as a consequence of Mostow Rigidity Theorem, see \cite{M}, and was also based in the work of Pansu, \cite{Pa}. Emanating from that point, quasiconformal mappings of more abstract spaces have started to be under study from various authors; for details, we refer to \cite{P1} and the references therein.

There are a lot of similarities as well as a lot of significant differences between the Kor\'anyi-Reimann theory of quasiconformal mappings in $\fH$ and the Ahlfors-Bers theory of quasiconformal mappings on the complex plane; the latter is instrumental for the study of Teichm\"uller spaces. Most notably, there exists an analogue of the measured Riemann Mapping Theorem which involves the existence of quasiconformal deformations. However, there is no existence theorem for the solution of the Beltrami equation and no uniqueness- existence theorems for extremal quasiconformal mappings, that is, mappings with constant maximal distortion.
In later developments of the Kor\'anyi-Reimann theory in the Heisenberg group, quasiconformal mappings that preserve $\mathbb{V}=\{0\}\times\R$ have appeared as generalisations to $\fH$ of classical quasiconformal mappings of the complex plane, see for instance \cite{BFP}. There, a Heisenberg stretch map is proved to be the minimiser of the mean distortion integral of mappings that map Kor\'anyi rings to Kor\'anyi rings in $\fH$. 

Mappings of $\fH$ that preserve $\mathbb{V}$ may be considered as  self mappings of $\C_*\times\R$; it is therefore natural to ask if there exist any particular structures in that set. Goldman noted among others (see  \cite{G}), that using the surjective map
$$
a:\C_*\times\R\ni(z,t)\mapsto -|z|^2+it\in\bH^1_\C,
$$
where $\bH^1_\C=\{\zeta\in\C\;|\;\Re(\zeta)<0\}$ is the left hyperbolic plane, one may pull back the hyperbolic metric $g_h$ of $\bH^1_\C$ to $\C_*\times\R$. However, this does not produce a Riemannian but rather a {\it semi}-Riemannian structure. Goldman also observed that we may identify $\C_*\times\R$ to $\partial(\bH^2_\C)\setminus\partial(\bH^1_\C)$ and there is a simply transitive action of the group ${\rm SU}(1,1)$. 

The {\it hyperbolic Heisenberg group} which was introduced in \cite{PS}, is the Lie group $(\C_*\times\R,\star)$ where 
$$
(z,t)\star(w,s)=(zw,t+s|z|^2).
$$ 
We note first that comparing $\fH^\star$ and $\fH$, we see that $\fH^\star$ is not a nilpotent Lie group. Actually, we show that $\fH^\star$ is isomorphic to the group $AN$ in the Iwasawa decomposition of ${\rm SU}(1,1)$ times ${\rm U}(1)$ (whereas, $\fH$ is the $N$ group in the Iwasawa decomposition of ${\rm SU}(2,1)$).

It is proved in \cite{PS} that $\fH^\star$ admits the Sasakian structure of the unit tangent bundle of hyperbolic plane. The contact form $\omega^\star$ of $\fH^\star$ is a multiple of $\omega$ when the latter is restricted to $\C_*\times\R$, therefore $(\fH^\star,\omega^\star)$ is a contact open submanifold of $(\fH,\omega)$. This implies that every contact transformation of $(\fH^\star,\omega^\star)$ is also a contact transformation of (a domain of) $(\fH,\omega)$. The sub-Riemannian tensor $g^\star_{cc}$ is the one first noted by Goldman as above, that is $g^\star_{cc}=\alpha^*g_h$ where $\alpha$ is the Kor\'anyi map and $g_h$ is the hyperbolic metric tensor in the hyperbolic plane.
The sub-Riemannian  tensor corresponds to the CC metric $d^\star_{cc}$ of $\fH^\star$. The group ${\rm SU}(1,1)\times{\rm U}(1)$ together with $j(z,t)\to(\overline{z},-t)$ is the group of its isometries; for details, see Section \ref{sec:hypheis}. \

Quasiconformal mappings in $\fH^\star$ can now be defined using the metric $d^\star_{cc}$. As in the case of the Heisenberg group, see \cite{KR2}, we are able to show that this definition is equivalent to an analytic as well to a geometric definition of a quasiconformal mapping in $\fH^\star$. This issue, together with the topic of quasiconformal deformations in $\fH^\star$ and the respective Measured Riemann Mapping Theorem are left out of this paper. Details about all that will appear elsewhere.  Here, we restrict ourselves to the study of smooth (at least of class $\cC^2$) contact quasiconformal transformations in $\fH^\star$, see Section \ref{sec:qcstar}. It is worth noting at this point that if $F:\fH^\star\to\fH^\star$ is a $K$-quasiconformal diffeomorphism, with maximal distortion function $K_F(z,t)$, then the maximal distortion of $F$ as well as its constant of quasiconformality $K$ coincide with the respective maximal distortion and constant of quasiconformality of $F$, if the latter is consider as a quasiconformal diffeomorphism between domains of the Heisenberg group $\fH$.  

We finally turn our attention to circles-preserving contact transformations of $\fH^\star$, that is, transformations that preserve the fibres of the Kor\'anyi map $\alpha$, see Section \ref{sec:cJ}. It turns out that these transformations have constant Jacobian determinant, which we can always assume to be 1. We then prove the Lifting Theorem \ref{thm:lift}: If $f$ is a symplectic $K$-quasiconformal self-mapping of the hyperbolic plane $\bH^1_\C$, then it can be lifted to a circles-preserving smooth $K$-quasiconformal transformation $F$ of the hyperbolic Heisenberg group $\fH^\star$. The proof of this theorem follows after showing the existence of solutions of a particular system of partial differential equations, see Lemma \ref{lem:psi}. Our Lifting Theorem is analogous to the well-known Lifting Theorem \ref{thm:liftC}, which states that if $f$ is a symplectic (with respect to the Euclidean K\"ahler form) $K$-quasiconformal self-mapping of the complex plane $\C$, then it can be lifted to a vertical lines-preserving smooth $K$-quasiconformal transformation $F$ of the Heisenberg group $\fH$; for details, see Section \ref{sec:liftC}.
Examples of lifted quasiconformal symplectic maps of $\bH^1_\C$ to quasiconformal maps of $\fH^\star$ are given in Section \ref{sec:examples}. In there, we rediscover the Heisenberg stretch and spiral maps which have been previously introduced in \cite{BFP} and \cite{P}, respectively.

\medskip

This paper is organised as follows: In Section \ref{sec:qcspc} we describe in brief the theory of quasiconformal mappings in strictly pseudoconvex domains of $\C^2$ that we use in this paper. Section \ref{sec:heis} is a rather detailed review on the Heisenberg group and the theory of smooth quasiconformal mappings in $\fH$. In Section \ref{sec:hypheis} we present the hyperbolic Heisenberg group in extent. Next, we define in Section \ref{sec:qcstar} quasiconformal diffeomorphisms of $\fH^\star$ and finally, in Section \ref{sec:lift} we prove Lifting Theorem \ref{thm:lift} and give some examples.

\section{Quasiconformal mappings in strictly pseudoconvex domains}\label{sec:qcspc}

The material of this section is standard. The book \cite{DT} is a standard reference for the theory of codimension 1 CR manifolds. For quasiconformal mappings in strictly pseudoconvex domains and in particular in the Heisenberg group,  we refer the reader to \cite{K, KR, KR1,KR2} for details.

A codimension $s$ {\rm CR} structure in a $(2p+s)$-dimensional real manifold $M$ is a pair $(\cH, J)$ where $\cH$ is a $2p$-dimensional smooth subbundle of the tangent bundle ${\rm T}(M)$ of $M$ and $J$ is an almost complex endomorphism of $\cH$ which is formally integrable: If $X$ and $Y$ are sections of $\cH$ then the same holds for 
  $\left[X, Y\right]-\left[JX, JY\right], \left[JX, Y\right]+\left[X, JY\right]$ and moreover, 
  $J(\left[X, Y\right]-\left[JX, JY\right])=\left[JX, Y\right]+\left[X, JY\right]$.

If $s=1$, $\cH$ may be defined as the kernel of a 1-form $\eta$, called the contact form of $M$, such that $\eta\wedge (d\eta)^{p}\neq 0$. The dependence of $\cH$ on $\eta$ is up to multiplication of $\eta$ by a nowhere vanishing smooth function. By choosing an almost complex structure $J$ defined in $\cH$ we obtain a {\rm CR} structure $(\cH, J)$ of codimension 1 in $M$. The subbundle $\cH$ is also called the horizontal subbundle of ${\rm T}(M)$. The closed form $d\eta$ endows $\cH$ with a symplectic structure and we may demand from $J$ to be such that $d\eta(X,JX)>0$ for each $X\in\cH$; we then say that $\cH$ is strictly pseudoconvex. The Reeb vector field $\xi$ is the vector field which satisfies
$
\eta(\xi)=1$ and $\xi\in\ker(d\eta).
$
By the contact version of Darboux's Theorem, $\xi$ is unique up to change of coordinates.

In dimension 3 and $s=1$, the nicest examples are those of strictly pseudoconvex {\rm CR} structures on boundaries of domains in $\C^2$. Let $D\subset \C^2$ be a domain with defining function $\rho:D\to\R_{>0}$, $\rho=\rho(z_1,z_2)$ and such that $d\rho\neq 0$ on $M=\partial D$. Each tangent space ${\rm T}_p(M)$, $p\in M$, contains a maximal $J$-subspace $\cH_p(M)$, where $J$ is the natural complex structure of $\C^2$. If $\partial$ is the holomorphic differential of $\C^2$ we consider the form $\tau=\partial\rho$. Then the kernel $\ker\tau$ of $\tau$ is generated from a (1,0) vector field $Z$ and if
$$
Z=X-iJX,
$$
then $\cH=\{X,JX\}$. On the other hand it is clear that $\cH=\ker\eta$ where
$$
\eta=\Im(\partial\rho)=\frac{1}{2}Jd\rho,
$$
Thus we obtain the {\rm CR} structure $(\cH=\ker(\eta),J)$.  This is a contact structure if and only if the Levi form 
$
L=d\eta=i\partial\overline{\partial}\rho
$
is positively oriented. A smooth curve $\gamma:[a,b]\to M$ is called horizontal if
$\dot\gamma(t)\in\cH_{\gamma(t)}$ for each $t$. The horizontal length of such a curve is then defined by
$$
\ell(\gamma)=\int_a^b\sqrt{L(\dot\gamma(t),\dot\gamma(t))}dt.
$$
If $p,q\in M$ we define the Carnot-Carath\'eodory ($CC$) distance $d_{cc}(p,q)$ of $p$ and $q$ as the infimum of the horizontal lengths of curves which join $p$ and $q$. Since $M$ is contact, such curves always exist. The $CC$ distance corresponds to the sub-Riemannian metric which is defined in $M$ by the relations
$$
\langle X,X\rangle=\langle JX,JX\rangle=1,\quad \langle X,JX\rangle=\langle JX,X\rangle=0.
$$ 
If $(S,d)$ is a metric space, let $F:S\to S$ be a self map of $S$. For $p\in S$ and $r>0$, consider the sets
$$
L_F(p,r)=\sup_{d(p,q)=r}d(F(p,F(q)),\quad\text{and}\quad l_F(p,r)=\inf_{d(p,q)=r}d(F(p,F(q)),
$$
as well as
$$
H_F(p)=\limsup_{r\to 0}\frac{L_F(p,r)}{l_F(p,r)},
$$
which is the distortion function of $F$. Then we say that $F$ is $K$-quasiconformal if $H_F(p)\le K$ for all $p$ in $S$. A 1-quasiconformal map shall be called conformal.

We may use the distance $d_{cc}$ to define quasiconformal mappings in $M$. If $F:M\to M$ is smooth (at least of class $\cC^2$) and orientation-preserving, then $F$ is $K$-quasiconformal if and only if
\begin{enumerate}
\item $F$ is a contact transformation, that is $F^*\eta=\lambda\eta$ for some positive function defined on $M$ and
\item for all $X\in \cH(M)$,
$$
\frac{\lambda}{K}L(X,X)\le L(F_*X,F_*X)\le\lambda K L(X,X).
$$ 
\end{enumerate} 
The above two conditions also serve as a definition of smooth $K$-quasiconformal mappings in arbitrary strictly pseudoconvex CR manifolds, see for instance \cite{KR}.

The complex dilation (or Beltrami coefficient) of a $K$-quasiconformal mapping $F:M\to M$ is defined by the symmetric antilinear operator $\mu_F:\cH^{(1,0)}(M)\to\cH^{(1,0)}(M)$ given by
 $$
 Z-\overline{\mu_FZ}\in F^{-1}_*(\cH^{(1,0)}(M)).$$
  This is well defined due to positive definiteness of the Levi form $L$. If $F:M\to M$ is contact, then it is $K$-quasiconformal if and only if
$$
\|\mu_F\|=\sup_{L(Z,Z)=1}L(\mu_FZ,\overline{\mu_FZ})^{1/2}\le\frac{K-1}{K+1}.
$$
There is a system of Beltrami equations involving the complex dilation $\mu_F$. By considering holomorphic coordinates $(z_1,z_2)$ in $\C^2$, we can write $F=(f_I,f_{II})$; the Beltrami system is then
$$
\overline{Z}f_k=\mu_F Zf_k,\quad k=I,II.
$$ 
 
\section{Heisenberg group}\label{sec:heis}
The material in this section can be found in various sources. For details about complex hyperbolic geometry and the Heisenberg group we refer to the book \cite{G}, as well as to the book \cite{CDPT}. The fundamental papers \cite{KR1, KR2} serve as a standard reference for the theory of quasiconformal mappings in the Heisenberg group.

\subsection{Complex hyperbolic plane}
In the concept of this paper, the complex hyperbolic plane $\bH^2_\C$ is the one point compactification of the Siegel domain
$
\calS=\{(z_1,z_2)\in\C^2\;|\;\rho(z_1,z_2)<0\},
$
where
\begin{equation}\label{eq:Sieg}
\rho(z_1,z_2)=2\Re(z_1)+|z_2|^2.
\end{equation}
The complex hyperbolic plane $\bH^2_\C$ is a complex manifold; there is a natural K\"ahler structure defined on $\bH^2_\C$ coming from the Bergmann metric:
\begin{equation}\label{eq:gyp}
ds^2=-\frac{4}{\rho}|dz_2|^2+\frac{4}{\rho^2}|\partial\rho|^2=\frac{4}{\rho}|dz_2|^2+\frac{4}{\rho^2}|dz_1+\overline{z_2}dz_2|^2.
\end{equation}
 The K\"ahler form is then
\begin{equation}\label{eq:Kchp1}
\Omega=-4i\partial\overline{\partial}(\log\rho)=-4i\left(-\frac{1}{\rho}dz_2\wedge d\overline{z_2}+\frac{1}{\rho^2}\partial\rho\wedge\overline{\partial}\rho\right).
\end{equation} 
The group of  holomorphic isometries is ${\rm PU}(2,1)$ (sometimes its three-fold cover ${\rm SU}(2,1)$ is also used). The action of  ${\rm PU}(2,1)$ is naturally extended to the boundary $\partial\bH^2_\C$. The hyperbolic plane sits in $\bH^2_\C$ as its complex submanifold $z_2=0$. Its isometry group is ${\rm PU}(1,1)$ (${\rm SU}(1,1)$).  

\subsection{The boundary}
The set $\partial\bH^2_\C\setminus\{\infty\}$ is identified in this manner to
$$
\partial\calS=\{(z_1,z_2)\in\C_2\;|\;\rho(z_1,z_2)=0\}.
$$
For clarity, we give a proof of the following well known result.
\begin{prop}
The boundary $\partial\calS$ of the Siegel domain $\calS$ admits a strictly pseudoconvex {\rm CR} structure. Its contact form is $\omega'=\Im(\partial\rho)$.
\end{prop}
\begin{proof}
The holomorphic differential $\partial\rho$ of $\rho$ is
$$
\partial\rho=dz_1+\overline{z_2}dz_2.
$$
Its kernel $\ker(\partial\rho)$ is generated by the $(1,0)$-vector field
$$
Z'=-\overline{z_2}\frac{\partial}{\partial z_1}+\frac{\partial}{\partial z_2},
$$
which defines the CR structure. The Levi form is
$$
\partial\overline{\partial}\rho=dz_2\wedge d\overline{z_2}.
$$
We have
$$
\partial\overline{\partial}\rho(Z',\overline{Z'})=dz_2\wedge d\overline{z_2}\left(-\overline{z_2}\frac{\partial}{\partial z_1}+\frac{\partial}{\partial z_2},\;-z_2\frac{\partial}{\partial \overline{z_1}}+\frac{\partial}{\partial \overline{z_2}}\right)=1>0,
$$
hence the Levi form is positively oriented in the CR structure, i.e., the CR structure is strictly pseudoconvex. The contact form is thus
$$
\omega=\Im(\partial\rho)=dy_1+\Im(\overline{z_2}dz_2).
$$
\end{proof}
\subsection{Heisenberg group: Lie algebra, contact structure}
The {\it Heisenberg group} $\fH$  is the set $\C\times\R$ with multiplication $*$ given by
$$
(z,t)*(w,s)=(z+w,t+s+2\Im(z\overline{w})),
$$
for every $(z,t)$ and $(w,s)$ in $\fH$. 
The Heisenberg group is a 2-step nilpotent Lie group. A left-invariant basis for its Lie algebra $\mathfrak{h}$ comprises the vector fields
\begin{eqnarray*}
X=\frac{\partial}{\partial x}+2y\frac{\partial}{\partial t},\quad Y=\frac{\partial}{\partial y}-2x\frac{\partial}{\partial t},
\quad T=\frac{\partial}{\partial t}.
\end{eqnarray*}
We also use the complex fields
$$
Z=\frac{1}{2}(X-iY)=\frac{\partial}{\partial z}+i\overline{z}\frac{\partial}{\partial t},\quad \overline{Z}=\frac{1}{2}(X+iY)=\frac{\partial}{\partial \overline{z}}-iz\frac{\partial}{\partial t}.
$$
The Lie algebra $\mathfrak{h}$ of $\fH$ has a grading $\mathfrak{h} = \mathfrak{v}_1\oplus \mathfrak{v}_2$ with
\begin{displaymath}
\mathfrak{v}_1 = \mathrm{span}_{\R}\{X, Y\}\quad \text{and}\quad \mathfrak{v}_2=\mathrm{span}_{\R}\{T\}.
\end{displaymath} 

In $\fH$ we consider the 1-form 
\begin{equation}\label{eq-omega}
\omega=dt+2xdy-2ydx=dt+2\Im(\overline{z}dz).
\end{equation}
The following proposition holds; it summarises well-known facts about $\fH$:

\medskip

\begin{prop}
Let the Heisenberg group $\fH$ together with  the 1-form $\omega$ be as in (\ref{eq-omega}). Then the
manifold $(\fH,\omega)$ is contact. Explicitly:
 \begin{enumerate}
  \item The form $\omega$ of $\fH$ is left-invariant.
  \item If $dm$ is the Haar measure for $\fH$ then $dm=-(1/4)\;\omega \wedge d\omega$.
  \item The kernel of $\omega$ is generated by 
$X$ and $Y$.
  \item The Reeb vector field for $\omega$ is $T$.
  \item The only non trivial Lie bracket relation is
  $
  [X,Y]=-4T.
  $
  \item Let $\cH=\ker(\omega)$ and consider the almost complex structure $J$ defined on $\cH$ by
 $JX=Y$, $ JY=-X.$
 Then  $J$ is compatible with $d\omega$ and moreover, $\cH$ is a strictly pseudoconvex {\rm CR} structure; that is, $d\omega$ is positively oriented on $\cH$. 
 \end{enumerate}
\end{prop}
Let $\Psi:\fH\to\partial\calS$ be the bijection given by
\begin{equation}\label{eq:Psi}
\Psi(z,t)=(-|z|^2+it,\sqrt{2}z).
\end{equation}
\begin{prop}
The map $\Psi:\fH\to\partial\calS$ is a CR diffeomorphism. 
 Also, $\Psi^*\omega'^{\star}=\omega$, where $\omega$ is the contact form of $\fH$ and $\omega'$ is the contact form of $\partial\calS$.
\end{prop}
\begin{proof}
We have
\begin{eqnarray*}
\Psi_*(Z)&=&Z(-|z|^2+it)\frac{\partial}{\partial z_1}+Z(-|z|^2-it)\frac{\partial}{\partial \overline{z_1}}
                                            +Z(\sqrt{2}z)\frac{\partial}{\partial z_2}
                                            +Z(\sqrt{2}\overline{z})\frac{\partial}{\partial \overline{z_2}}\\
                                            &=&\left(\frac{\partial}{\partial z}+i\overline{z}\frac{\partial}{\partial t}\right)(-|z|^2+it)\frac{\partial}{\partial z_1}+
                                            \left(\frac{\partial}{\partial z}+i\overline{z}\frac{\partial}{\partial t}\right)(\sqrt{2}z)\frac{\partial}{\partial z_2}\\
                                            &=&-\sqrt{2}\overline{z_2}\frac{\partial}{\partial z_1}+\sqrt{2}\frac{\partial}{\partial z_2}=\sqrt{2}Z',                                          
\end{eqnarray*}
which proves that $\Psi$ is CR. The last statement of the proposition is obvious.
\end{proof}

\subsection{Sub-Riemannian structure, isometries and similarities}
The sub-Riemannian product of $\fH$ is defined by the relations
\begin{equation*}\label{eq:subR}
\langle X,X\rangle=\langle Y,Y\rangle=1,\quad\langle X,Y\rangle=\langle Y,X\rangle=0.
\end{equation*}
A dual basis for $\{X,Y,T\}$ comprises of the forms
$$
\phi=dx,\quad \psi=dy,\quad \omega.
$$
The sub-Riemannian metric tensor is thus given by
$$
g_{cc}=\phi^2+\psi^2=dx^2+dy^2.
$$
This defines a K\"ahler structure on the horizontal tangent bundle $\cH$. The group ${\rm Isom}(g_{cc})$ of sub-Riemannian (CC) isometries, comprises compositions of the following transformations:
\begin{itemize}
\item Left (Heisenberg) translations: For fixed $(w,s)\in\fH$,
$$
L_{(w,s)}(z,t)=(w,s)*(z,t).
$$
\item Rotations around the vertical axis: For $\theta\in\R$,
$$
R_\theta(z,t)=(ze^{i\theta},t).
$$
\item Conjugation:
$$
j(z,t)=(\overline{z},-t).
$$
\end{itemize}
All the above isometries are orientation-preserving. The group ${\rm Sim}(g_{cc})$ of sub-Riemannian similarities is ${\rm Isom}(g_{cc})$ together with the group of
\begin{itemize}
\item Dilations: For $\delta>0$, $\delta\neq 1$,
$$
D_\delta(z,t)=(\delta z,\delta^2 t).
$$
\end{itemize}
Thus we may make the following identifications:
$$
{\rm Isom}(g_{cc})=\fH\times{\rm U}(1)\times\Z_2,\quad {\rm Sim}(g_{cc})={\rm Isom}(g_{cc})\times(\R_{>0},\cdot). 
$$
\subsection{Horizontal curves}

A smooth curve $\gamma:[a,b]\to \fH$ with 
\begin{displaymath}
\gamma(s)=(z(s),t(s))\in\mathbb{C}\times \mathbb{R},
\end{displaymath}
 is called {\it horizontal} if $\dot{\gamma}(s)\in \cH_{\gamma(s)}(\fH)$, $s\in [a,b]$; equivalently,
 \begin{displaymath}
 \dot t(s)=-2\Im\left(\overline{z(s)}\dot z(s)\right),\quad s\in [a,b].
\end{displaymath}
The horizontal length $\ell_\fH(\gamma)$ of $\gamma$ with respect to the sub-Riemannian metric is given by
\begin{displaymath}
 \ell_\fH(\gamma)=\int_a^b\sqrt{\langle\dot\gamma,X\rangle^2_{\gamma(s)}+\langle\dot
 \gamma,Y\rangle_{\gamma(s)}^2}ds=\int_a^b |\dot z (s)|\;ds.
\end{displaymath}
Therefore the horizontal length of $\gamma$ is equal to the euclidean length of $\tilde\gamma=\Pi\circ\gamma$, where $\Pi:\fH\to\C$ is the projection $(z,t)\mapsto z$. The fibres of $\Pi$ are straight lines perpendicular to the complex plane and $(\Pi;\fH,\C,\R)$  is a line bundle.

Recall the following definition, see for instance \cite{B}:
\begin{defn}
Let $\pi:M\to B$ be a submersion and $\cH$ a distribution on $M$ supplementary to $\cV=\ker(\pi_*)$. The distribution $\cH$ is {\it Ehresmann complete} if for any path $\gamma^*$ in $B$ with starting point $p^*$, and any $p\in\pi^{-1}(p^*)$, there exists a horizontal lift $\gamma$ of $\gamma^*$ in $M$ ($\pi\circ\gamma=\gamma^*$), starting from $p$.
\end{defn}
The distribution $\cH$ of $\fH$ is Ehressmann complete: Indeed,
if $\tilde\gamma(s)=z(s)$, $s\in[a,b]$, $\tilde\gamma(a)=z_0$, is a plane curve, then the fibre of $z_0$ is
$$
\Pi^{-1}(z_0)=\{(z_0,t)\;|\;t\in\R\}.
$$
Let $t_0\in\R$. We define $\gamma:[a,b]\to\fH$, $\gamma(s)=(z(s),t(s))$, where
$$
t(s)-t_0=-2\int_a^s\Im(\overline{z(u)}\dot z(u))du.
$$
Clearly $\gamma$ emanates from $(z_0,t_0)$ and it is horizontal. In the case where $\tilde\gamma$ is closed, $\tilde\gamma(a)=\tilde\gamma(b)=z_0$, we have
\begin{eqnarray*}
t(s)-t_0&=&-2\int_{\gamma}xdy-ydx\\
&=&-4\iint_{{\rm int}(\gamma)}dx\wedge dy=-4{\rm Area}({\rm int}(\gamma)),
\end{eqnarray*}
were we have applied Green's theorem and have assumed that all orientations are positive. Also here, ${\rm Area}({\rm int}(\gamma))$ is the Euclidean area of ${\rm int}(\gamma)$.

\subsection{CC and Kor\'anyi-Cygan distances}

Since $\fH$ is contact, given any two arbitrary but distinct points $p=(z,t)$ and $q=(z',t')$ in $\fH$, there exists a horizontal curve joining $p$ and $q$. We define the {\it Carnot-Carath\'eodory distance} of two arbitrary points $p,q\in\fH$ to be
$$
d_{cc}(p,q)=\inf_\gamma\ell_\fH(\gamma),
$$ 
where $\gamma$ is horizontal and joins $p$ and $q$. Clearly,
$$
{\rm Isom}(d_{cc})=\fH\times{\rm U}(1)\times\Z_2,\quad {\rm Sim}(d_{cc})={\rm Isom}(g_{cc})\times(\R_{>0},\cdot). 
$$
Besides the $d_{cc}$ there is another distance function which is naturally defined in $\fH$. This is the  {\it Kor\'anyi--Cygan}  (or {\it Heisenberg}) distance $d_\fH$ which is given by the relation
$$
d_\fH\left((z_1,t_1),\,(z_2,t_2)\right)
=\left|(z_1,t_1)^{-1}*(z_2,t_2)\right|,
$$
where $|(z,t)|=|-|z|^2+it|^{1/2}$. This is not a path distance and does not come from any Riemannian metric defined in $\fH$. However, $d_\fH$ is invariant under left translations, rotations around the vertical axis and conjugation, so that
$$
{\rm Isom}(d_\fH)={\rm Isom}(d_{cc}).
$$ 
As for the similarity group ${\rm Sim}(d_\fH)$, besides dilations it also comprises
{\rm inversion} $I$. This is defined in $\fH\setminus\{(0,0)\}$ by 
$$
I(z,t)=\left(\frac{z}{-|z|^2+it},\;
-\frac{t}{|-|z|^2+it|^2}\right).
$$
In comparison, the Kor\'anyi-Cygan metric and the Carnot-Carath\'eodory metric are equivalent metrics, they generate the same infinitesimal structure in $\fH$, their isometry groups are the same, but not their similarity groups: inversion is not a similarity of $d_{cc}$.
\subsection{Smooth contact diffeomorphisms of $\fH$}
A {\it contact transformation} $F:\fH\to\fH$ is a smooth (at least of class $\cC^2$) diffeomorphism which preserves the contact structure, i.e.
\begin{equation}\label{eq:contact_form}
 F^*\omega = \lambda \omega,
\end{equation}
for some non-vanishing real valued function $\lambda$. We  write $F=(f_I,f_3)$, $f_I=f_1+\mathrm{i}f_2$. The matrix of its differential $F_*$ is
$$
DF=\left[\begin{matrix}
Zf_I&\overline{Z}f_I&Tf_I\\
\\
Z\overline{f_I}&\overline{Zf_I}&T\overline{f_I}\\
\\
0&0&\lambda
\end{matrix}\right].
$$
A contact mapping  $F$ is completely determined by $f_I$ in the sense that the contact condition  (\ref{eq:contact_form}) is equivalent to the following system of partial differential equations:
\begin{eqnarray}
\notag\label{eq:C1} 
&&\overline{f}_IZ f_I - f_I Z \overline{f}_I +iZf_3=0,\\
&&
\label{eq:C2}
f_I \overline{Z}\overline{f}_I-\overline{f}_I \overline{Z}f_I -
i\overline{Z}f_3=0,\\
&&
\notag\label{eq:C3}
-i(\overline{f}_I T f_I - f_I T \overline{f}_I +
iTf_3)=\lambda.
\end{eqnarray}
If $F$ is $\cC^2$ it is  proved that the Jacobian determinant $J_F$ satisfies
$J_F=\lambda^2.
$
 Moreover, for orientation-preserving contact transformations we have
\begin{equation}\label{eq:D0}
J_0=\det(D_0)=\det\left[\begin{matrix}
Zf_I&\overline{Z}f_I\\
Z\overline{f_I}&\overline{Zf_I}
\end{matrix}\right]=\lambda.
\end{equation}

Let $|\cdot|$ be the sub-Riemannian norm on $\fH$ and let $F=(f_I,f_3)$ be a smooth, orientation-preserving contact transformation, $F^*\omega=\lambda\omega$. Let also $p\in\fH^\star$. A horizontal vector $V_p\in\cH_p$ is of the form $a\bX_p+b\bY_p$ and $|V_p|=\sqrt{a^2+b^2}$. Since $F$ is contact, for every such $V_p$ we always have $F_{*,p}V_p\in\cH_{F(p)}$. We consider
\begin{equation*}\label{eq:lambda12}
\lambda_1(p)=\sup_{|V_p|=1}|F_{*,p}V_p|,\quad \lambda_2(p)=\inf_{|V_p|=1}|F_{*,p}V_p|.
\end{equation*}
We have
$$
\lambda_1(p)=|Zf_I(p)|+|\overline{Z}f_I(p)|
$$
and
$$
\lambda_2(p)=|Zf_I(p)|-|\overline{bZ}f_I(p)|.
$$
Let $D_0(p)$ as in (\ref{eq:D0})
and consider the matrix $D_0(p)^*D_0(p)$, where $D_0(p)^*$ is the complex transpose of $D_0(p)$. Note that $\lambda_1^2(p),\lambda_2^2(p)$ are the eigenvalues of $D_0(p)^*D_0(p)$ and we have also proved that $J_0(p)=\det(J_0(p))=\lambda(p)$. We thus have:

\begin{prop}
If $F$ is an orientation-preserving, smooth contact transformation of $\fH$, then we have
$$
\lambda_1\lambda_2=\lambda.
$$
\end{prop} 

\begin{defn}
Let $F=(f_I,f_3)$ a smooth, orientation-preserving contact transformation of $\fH$. The {\it maximal distortion function} $K_F$ is defined for each $p\in\fH$ by
$$
K_F(p)=\frac{\lambda_1(p)}{\lambda_2(p)}=\frac{|Zf_I)(p)|+|\overline{Z}f_I(p)|}{|Zf_I(p)|-|\overline{Z}f_I(p)|}.
$$
\end{defn}

\subsection{Quasiconformal contact diffeomorphisms of $\fH$}
\begin{defn}
Let $F=(f_I,f_3)$ be a smooth, orientation-preserving contact transformation of $\fH$. If its maximal distortion function satisfies
$$
K_F(p)\le K, \quad p\in\fH,
$$
for some $K$, then $F$ shall be called $K$-quasiconformal.
\end{defn}
The {Beltrami coefficient} $\mu_F:\fH\to\C$ of the $K$-quasiconformal mapping $F$ is given by
$$
\mu_F(p)=\frac{\overline{Z}f_I}{Z f_I}(p),\quad p\in\fH.
$$
If $F$ is $K$-quasiconformal then $F$ equivalently satisfies
$$
\|\mu_F\|={\rm esssup}_{p\in\fH}\{|\mu_F(p)|\}\le k=\frac{K-1}{K+1}\in[0,1).
$$
The contact conditions also imply
\begin{eqnarray*}
&&
\frac{\overline{Z}f_I}{Z f_I}=\frac{\overline{Z}f_{II}}{Z f_{II}}=\mu_F,
\end{eqnarray*}
where $f_{II}=-|f_I|^2+if_3$.
\subsection{Lifting Theorem}\label{sec:liftC}

The following theorem is a version of Theorem 5.3 of \cite{CT}, see also Theorem 1.6 of \cite{BHT}, adapted to our setting. For completeness, we give a proof.

\begin{thm}\label{thm:liftC}
Suppose $f:\C\to\C$ is symplectic with respect to the Euclidean K\"ahler form of $\C$. Then there exists a function $\phi:\C\to\R$ such that the mapping $F:\fH\to\fH$ given for each $(z,t)\in\fH$ by
$$
F(z,t)=(f(z),\;t+\phi(z))
$$
is a contact transformation of $\fH$ with $\lambda=1$. If moreover $f$ is quasiconformal with Beltrami coefficient $\mu_f$, then $F$ is also quasiconformal with
$$
\mu_F(z,t)=\mu_f(z).
$$
\end{thm}
\begin{proof}
We define the form $\eta$ in $\C$ by
$$
\eta=2\Im(\overline{z}dz-\overline{f}df).
$$
Since $f$ is symplectic, $J_f=|f_z|^2-|f_{\overline{z}}|^2=1$ and hence $\eta$ is closed, $d\eta=0$.

By Poincar\'e's Lemma, there exists  a function $\phi:\C\to\R$ such that $\eta=d\phi$. Therefore, 
$$
F^*\omega=dt+d\eta+2\Im(\overline{f}df)=\omega
$$
and $F$ is contact. On the other hand,
$$
\mu_F(z,t)=\frac{\overline{Z}f}{Z f}=\frac{f_{\overline z}}{f_z}=\mu_f (z).
$$
\end{proof}
Note that the map $F$ given in Theorem \ref{thm:liftC} is a {\it vertical lines-preserving} map. That is, it preserves the fibres of the projection $\Pi:\fH\to\C$. It can be shown that vertical lines-preserving maps are the only contact transformations of $\fH$ with constant $\lambda$. For more details, see \cite{CT}.

\section{Hyperbolic Heisenberg group}\label{sec:hypheis}
Hyperbolic Heisenberg group $\fH^\star$ has been introduced in \cite{PS}. In Section \ref{sec:dlg} we give the definition of $\fH^\star$ and describe its Lie group structure. In Section \ref{sec:mm} we describe $\fH^\star$ as a subgroup of ${\rm SU}(1,1)\times{\rm U}(1)$, as the latter sits in the set ${\rm PU}(2,1)$ of holomorphic isometries of complex hyperbolic plane. Section \ref{sec:SPCR} is devoted to the strictly pseudoconvex structure of $\fH^\star$ which is actually the one coming from the set $\bH^\star=\partial\bH^2_\C\setminus\partial\bH^1_\C$. Finally, in Sections \ref{sec:sr} and \ref{sec:hc} we study the sub-Riemannian metric and the horizontal curves in $\fH^\star$, respectively. 

\subsection{Definition and Lie group structure}\label{sec:dlg}

\begin{defn}
The hyperbolic Heisenberg group $\fH^\star$ is $\C_*\times\R$ with multiplication rule
$$
(z,t)\star(w,s)=(zw,t+s|z|^2).
$$
\end{defn}

\begin{prop}
The hyperbolic Heisenberg group $\fH^\star$:
\begin{itemize}
\item is a non-Abelian group. Its unit element of $\fH^\star$ is (1,0) and the inverse of an arbitrary $(z,t)\in\fH^\star$ is $(1/z,-t/|z|^2)$;
\item is a Lie group with underlying manifold $\C_*\times\R$. The map
\begin{eqnarray*}
&&
\fH^\star\times\fH^\star\ni\left((z,t),(w,s)\right)\mapsto (z,t)^{-1}\star (w,s)=\left(\frac{w}{z},\frac{-t+s}{|z|^2}\right)\in\fH^\star,
\end{eqnarray*}
is smooth. 
\end{itemize}
\end{prop}

In order to detect a basis for the left invariant vector fields of $\fH^\star$, 
we fix a left translation
$$
F(z,t)=L_{(w,s)}(z,t)=(wz,s+t|w|^2).
$$
The complex matrix $DF$ of the differential $F_{*}$ is:
$$
DF=\left[\begin{matrix}
w&0&0\\
0&\overline{w}&0\\
0&0&|w|^2
\end{matrix}\right].
$$
Hence the vector fields
\begin{equation}\label{eq:basisc}
Z^{*}=z\frac{\partial}{\partial z},\quad \overline{Z}^{*}=\overline{z}\frac{\partial}{\partial \overline{z}},\quad T^{*}=|z|^2\frac{\partial}{\partial t},
\end{equation}
form a left-invariant basis for the Lie algebra of $\fH^{\star}$ and the corresponding real basis is
\begin{equation}\label{eq:basis}
X^{*}=x\frac{\partial}{\partial x}+y\frac{\partial}{\partial y},\quad Y^{*}=x\frac{\partial}{\partial y}-y\frac{\partial}{\partial x},\quad T^{*}=(x^2+y^2)\frac{\partial}{\partial t},
\end{equation}
so that
$$
Z^{*}=\frac{1}{2}(X^{*}-iY^{*}),\quad \overline{Z}^{*}=\frac{1}{2}(X^{*}+iY^{*}).
$$
We have the bracket relations
\begin{equation}\label{eq:nontrivb}
[X^{*},T^{*}]=2T^{*},\quad [X^*,Y^*]=[Y^*,T^*]=0.
\end{equation}
Also, since $\det(DF)=|w|^4$ we have that the Haar measure of $\fH^{\star}$ is given by
$$
dm=\frac{dx\wedge dy\wedge dt}{(x^2+y^2)^2}=\frac{i}{2}\cdot\frac{dz\wedge d\overline{z}\wedge dt}{|z|^2}.
$$
\subsection{Matrix model, isomorphisms and actions}\label{sec:mm}
We next show that there exists an identification of $\fH^\star$ with a $3\times 3$ matrix group which is a subgroup of ${\rm PU}(2,1)$, that is, the set of holomorphic isometries of complex hyperbolic plane. 
Let $M$ be the map from $(\C\times\R)\setminus\V$ to ${\rm GL}(3,\C)$ given for each $p=(z,t)\in(\C\times\R)\setminus\V$ by
 \begin{equation}\label{eq:T}
 M(z,t)=\left[\begin{matrix}
 |z|&0& i\frac{t}{|z|}\\
 \\
 0&e^{i\arg z}&0\\
 \\
 0&0&\frac{1}{|z|}\end{matrix}\right].
 \end{equation}
The matrix $M(z,t)$ is an element of ${\rm PU}(2,1)$: If
$$
\left[\begin{matrix}
 0&0&1\\
 \\
 0&1&0\\
 \\
 1&0&0\end{matrix}\right],
$$
then $M(z,t)J(M(z,t))^*=J$.

The proof of the following proposition is straightforward; it reveals the multiplication rule in the group $\fH^\star$.
\begin{prop}
The map $M$ as in (\ref{eq:T}) is a group homomorphism from $\fH^\star=(\C_*\times\R,\star)$ to $({\rm PU}(2,1),\cdot)$ where $\cdot$ is the usual matrix multiplication.
\end{prop}
\begin{proof}
For $(z,t)$ and $(z',t')$ in $\C_*\times\R$ we have
\begin{eqnarray*}
M(z,t)\cdot M(z',t')&=&\left[\begin{matrix}
 |z|&0& i\frac{t}{|z|}\\
 \\
 0&e^{i\arg z}&0\\
 \\
 0&0&\frac{1}{|z|}\end{matrix}\right]\cdot\left[\begin{matrix}
 |z'|&0& i\frac{t'}{|z'|}\\
 \\
 0&e^{i\arg z'}&0\\
 \\
 0&0&\frac{1}{|z'|}\end{matrix}\right]\\
 &=&\left[\begin{matrix}
 |zz'|&0& i\frac{t+t'|z|^2}{|zz'|}\\
 \\
 0&e^{i\arg zz'}&0\\
 \\
 0&0&\frac{1}{|zz'|}\end{matrix}\right]=M((z,t)\star(z',t')).
\end{eqnarray*}
\end{proof}
The group $\fH^\star$ is therefore identified to the subgroup of ${\rm PU}(2,1)$ comprising elements of the form
$$
M(r,\theta,s)=\left[\begin{matrix}
 r&0& i\frac{s}{r}\\
 \\
 0&e^{i\theta}&0\\
 \\
 0&0&\frac{1}{r}\end{matrix}\right],\quad r>0,\;\theta\in\R,\;s\in\R.
$$
Let
$$
{\rm SU}(1,1)=\left\{\left[\begin{matrix}
a&ib\\
ic&d
\end{matrix}\right],\;a,b,c,d\in\R,\;ad+bc=1\right\},
$$
be the double cover of the set of holomorphic isometries of the complex hyperbolic line $\bH^1_\C$ which here is represented by the left hyperbolic plane
$$
\calL=\{\zeta=\xi+i\eta\in\C\:|\;xi<0\}.
$$
We consider the following embedding of ${\rm SU}(1,1)\times {\rm U}(1)$ into ${\rm PU}(2,1)$:
\begin{equation}\label{eq:embedding}
\left(\left[\begin{matrix}
a&ib\\
ic&d
\end{matrix}\right],\;e^{i\theta}\right)\mapsto\left[\begin{matrix}
a&0&ib\\
\\
0&e^{i\theta}&0\\
\\
ic&0&d
\end{matrix}\right].
\end{equation}
\begin{cor}
The image of the map $M$ as in (\ref{eq:T}) is in ${\rm SU}(1,1)\times {\rm U}(1)$. 
\end{cor}
Recall the Iwasawa $KAN$ decomposition of ${\rm SU}(1,1)$: 
$$
K={\rm SO}(2)\simeq {\rm U}(1),\quad A\simeq(\R_+,\cdot),\quad N=(\R,+).
$$
If
$$
A=\left(\begin{matrix}
a&ib\\
ic&d\end{matrix}\right)\in {\rm SU}(1,1,)
$$
there exist numbers $r>0$, $\theta\in\R$ and $t\in\R$ such that
\begin{equation*}
A=\left[\begin{matrix}
\cos\theta&i\sin\theta\\
i\sin\theta&\cos\theta
\end{matrix}\right]\cdot
\left[\begin{matrix}
r&0\\
0&1/r
\end{matrix}\right]\cdot
\left[\begin{matrix}
1&t\\
0&1
\end{matrix}\right].
\end{equation*}
From this and the above discussion we conclude that $\fH^\star$ is isomorphic to $AN\times{\rm U}(1)$. But $AN$ is just the group ${\mathcal Aff}$ which realizes the hyperbolic plane $\bH^1_\C$. We therefore have:
\begin{prop}
The hyperbolic Heisenberg group $\fH^\star$ is in bijection with $\calL\times S^1$.
\end{prop}
\begin{cor}
$$
\fH^\star=({\rm SU}(1,1)\times{\rm U}(1))/{\rm SO}(2).
$$
\end{cor}

The set $\C_*\times\R$ is the set $\bH^\star=\partial\bH^2_\C\setminus\partial\bH^1_\C$ which is a component of the boundary of $\bH^2_\C\setminus\bH^1_\C$. We shall $\bH^\star$  the {\it truncated boundary of complex hyperbolic plane}. We shall show in Section \ref{sec:SPCR} that $\fH^\star$ and $\bH^\star$ have the same relationship with that of the Heisenberg group $\fH$ and $\partial\bH^2_\C\setminus\{\infty\}$.

For the moment we focus in the left action of $\fH^\star$ on $\bH^\star$. This is given by the rule: 
$$
(M(r,s,\theta),(z,t))\mapsto(re^{i\theta}z,r^2t+s);
$$
\begin{prop}
This action of $\fH^\star$ to $\bH^\star$ is free and simply transitive.
\end{prop}
\begin{proof} 
Let $f=M(r,s,\theta)$ in $\fH^\star$ and let $p=(z,t)\in\C_*\times\R$ such that $fp=p$. Then $fp=p$ implies $f$ is the identity element of $\fH^\star$.

On the other hand, $M(|z|,\arg z,t)$ fixes infinity and maps $(1,0)$ to $(z,t)$; that is, the action is simply transitive.
\end{proof}
\begin{rem}
It is worth noting that this action is also extended to a free action on $\bH_\C^2\setminus\bH_\C^1$, where here 
$$
\bH_\C^1=\{(z_1,z_2)\in\bH_\C^2\;|\;z_2=0\}.
$$
The rule is:
$$
(M(r,s,\theta),(z_1,z_2))\mapsto(r^2z_1+is,re^{i\theta}z_2).
$$
This action is not simply transitive on the set $\overline{\bH_\C^2\setminus\bH_\C^1}$: For $(z_1,z_2)\in\bH_\C^2\setminus\bH_\C^1$ there is no $M(r,s,\theta)$ mapping $(-1,\sqrt{2})$ to $(z_1,z_2)$. 
\end{rem}

\subsection{Strictly pseudoconvex CR structure}\label{sec:SPCR}
From the bracket relations (\ref{eq:nontrivb}) we obtain only integrable CR structures.
In order to detect a strictly pseudoconvex CR structure for $\fH^\star$, we  define a
distinguished basis for the Lie algebra.
We consider the vector fields
\begin{equation}\label{eq-XYT}
\bX=X^*,
\quad \bY=Y^*-2T^*,
\quad\bT=Y^*.
\end{equation}
They are left-invariant and form a basis for the tangent space of $\fH^{\star}$. The only non-trivial Lie bracket relation between $\bX,\bY$ and $\bT$ is
$$
[\bX,\bY]=2(\bY-\bT).
$$
We shall also use the complex vector fields
$$
\bZ=\frac{1}{2}(\bX-i\bY),\quad\overline{\bZ}=\frac{1}{2}(\bX+i\bY).
$$
Notice that
$$
\bZ=zZ,\quad \overline{\bZ}=\overline{z}\overline{Z},
$$
where $Z$ and $\overline{Z}$ are the complex left-invariant vector fields of the Heisenberg group $\fH$.

From the contact structure of the form $\omega$ of the Heisenberg group $\fH$ we obtain a left-invariant contact form for the hyperbolic Heisenberg group $\fH^\star$ which we shall denote by $\omega^\star$. The following proposition describes this form.

\begin{prop}\label{prop:omegasta}
We consider the following 1-form in the hyperbolic Heisenberg group $\fH^\star$:
\begin{equation*}
\omega^\star=\frac{\omega}{2|z|^2},
\end{equation*}
where $\omega$ is the restriction to $\C_*\times\R$ of the contact form $\omega$ of the Heisenberg group $\fH$. Then the
manifold $(\fH^\star,\omega^\star)$ is contact. Explicitly:
 \begin{enumerate}
  \item[{(i)}] The form $\omega^\star$ of $\fH^\star$ is left-invariant.
  \item[{(ii)}] If $dm$ is the Haar measure for $\fH^\star$, then $dm=-\omega^\star \wedge d\omega^\star$.
  \item[{(iii)}] The kernel of $\omega^\star$ is 
$$
\ker(\omega^\star)=\langle\bX,\bY\rangle.
$$  
  \item[{(iv)}] The Reeb vector field for $\omega^\star$ is $\bT$.
  \item[{(v)}] Let $\cH^\star=\ker(\omega^\star)$ and consider the almost complex structure $J$ defined on $\cH$ by
 $$
 J\bX=\bY,\quad J\bY=-\bX.
 $$
 Then  $J$ is compatible with $d\omega^\star$ and moreover, $\cH^\star$ is a strictly pseudoconvex {\rm CR} structure; that is, $d\omega^\star$ is positively oriented on $\cH^\star$. 
 \end{enumerate}
\end{prop}

\begin{proof}
\noindent (i) We fix an arbitrary $(w,s)\in\fH^\star$ and consider the left translation
$$
F(z,t)=L_{(w,s)}(z,t)=(wz,s+t|w|^2).
$$
Then
\begin{eqnarray*}
&&
 F^*(\omega^\star)=\frac{d(s+t|w|^2)+2\Im(\overline{wz}d(wz))}{2|w|^2|z|^2}
 =\frac{|w|^2\left(dt+2\Im(\overline{z}dz)\right)}{2|w|^2|z|^2}=\omega^\star.
\end{eqnarray*}
\noindent (ii) We may write
$$
\omega^*=\frac{dt}{2|z|^2}+d\arg z,
$$
so that
$$
d\omega^*=2\left(\frac{d(|z|^2)}{2|z|^2}\right)\wedge \left(-\frac{dt}{2|z|^2}\right).
$$
We calculate
$$
\omega^\star\wedge d\omega^\star=\frac{dt\wedge dx\wedge dy}{|z|^4}=-dm.
$$
\noindent (iii) is obvious.
 
\noindent (iv)  Clearly, 
$\omega^\star(\bT)=1$ and it is straightforward to show that $\bT\in\ker(d\omega^\star)$. 

\noindent (v) we first observe that the dual basis to $\{\bX,\bY,\bT\}$ is $\{\phi^*,\psi^*,\omega^\star\}$,
where
\begin{equation}\label{eq:basis-cotr}
\phi^\star=\frac{d(|z|^2)}{2|z|^2},\quad \psi^\star=-\frac{dt}{2|z|^2}.
\end{equation}
In this basis,
$$
d\phi^\star=0,\quad d\psi^\star=-2\;\phi^*\wedge\psi^*,\quad d\omega^\star=2\;\phi^\star\wedge\psi^\star.
$$
For the symplectic form $d\omega^\star$ we thus have
$$
d\omega^\star(\bX,\bY)=2,\quad d\omega^\star(\bX,\bT)=d\omega^\star(\bY,\bT)=0.
$$
\end{proof}

We now consider the truncated boundary of complex hyperbolic plane
$
\bH^{\star}=\partial\bH^2_\C\setminus\partial\bH^1_\C.
$
The set $\bH^{\star}$ comprises of points $(z_1,z_2)$ of $\C^2$ such that
$$
\rho^{\star}(z_1,z_2)=\frac{\rho(z_1,z_2)}{|z_2|^2}=0,
$$
where $\rho$ is the defining function of $\bH^2_\C$ as in (\ref{eq:Sieg}). Let $\Psi:\fH^{\star}\to\bH^{\star}$ be the bijection given by
$$
\Psi(z,t)=(-|z|^2+it,\sqrt{2}z).
$$
\begin{prop}
There is a strictly pseudoconvex CR structure on $\bH^{\star}$ and the map $\Psi$ is CR. Also, if $\eta^{\star}$ is the corresponding contact form, then $\Psi^*\eta^{\star}=\omega^\star$, where $\Psi$ is as above and $\omega^\star$ is the contact form of the hyperbolic Heisenberg group $\fH^\star$.
\end{prop}
\begin{proof}
It follows from $\rho^{\star}(z_1,z_2)=\frac{\rho(z_1,z_2)}{|z_2|^2}$ that
$$\partial \rho^{\star}=\frac{dz_1+\overline{z_2}dz_2}{|z_2|^2},\qquad 
\bar{\partial} \rho^{\star}=\frac{d\overline{z_1}+z_2d\overline{z_2}}{|z_2|^2}.$$
A CR structure is defined by the (1, 0) vector field
$Z=-|z_2|^2\frac{\partial}{\partial z_1}+z_2\frac{\partial}{\partial z_2}$ which generates $\ker(\partial\rho^\star)$. Direct calculations show that the Levi form is
$$\partial\bar{\partial} \rho^{\star}=-\frac{\overline{z_2}}{|z_2|^4}dz_2\wedge d\overline{z_1},$$ therefore
$$
\partial\bar{\partial}{ \rho}^{\star}(Z, \bar{Z})=-\frac{\overline{z_2}}{|z_2|^4}dz_2\wedge d\overline{z_1}\left(-|z_2|^2\frac{\partial}{\partial z_1}+z_2\frac{\partial}{\partial z_2}, -|z_2|^2\frac{\partial}{\partial \overline{z_1}}+\overline{z_2}\frac{\partial}{\partial \overline{z_2}}\right)\\
=1>0.
$$
It follows that the Levi form is positively oriented on the CR structure. Now,
\begin{eqnarray*}
\Psi_{\ast}(\bZ)&=&\bZ(-|z|^2+it)\frac{\partial}{\partial z_1}+\bZ(-|z|^2-it)\frac{\partial}{\partial \overline{z_1}}
                                            +\bZ(\sqrt{2}z)\frac{\partial}{\partial z_2}
                                            +\bZ(\sqrt{2}\overline{z})\frac{\partial}{\partial \overline{z_2}}\\
                                            &=&\left(z\frac{\partial}{\partial z}+i|z|^2\frac{\partial}{\partial t}\right)(-|z|^2+it)\frac{\partial}{\partial z_1}+
                                            \left(z\frac{\partial}{\partial z}+i|z|^2\frac{\partial}{\partial t}\right)(\sqrt{2}z)\frac{\partial}{\partial z_2}\\
                                            &=&-|z_2|^2\frac{\partial}{\partial z_1}+z_2\frac{\partial}{\partial z_2}=Z                                                                          
\end{eqnarray*}
and thus $\Psi$ is CR. On the other hand, the contact form of $\bH^\star$ is
$$
\eta^{\star}=\Im(\partial \rho^{\star})=\frac{dy_1+\Im(\overline{z_2}dz_2)}{|z_2|^2}
$$
and we have $\Psi^*\eta^\star=\omega^\star$.
The proof is complete.
\end{proof}
\subsection{Sub-Riemannian metric}\label{sec:sr}
The sub-Riemannian product is defined by the relations
\begin{equation}\label{eq:subRstar}
\langle \bX,\bX\rangle=\langle \bY,\bY\rangle=1,\quad\langle \bX,\bY\rangle=\langle \bY,\bX\rangle=0.
\end{equation}
The corresponding tensor is
$$
g^\star_{cc}=(\phi^\star)^2+(\psi^\star)^2=\frac{(d(|z|^2))^2+dt^2}{4|z|^4},
$$
and defines a K\"ahler structure on the horizontal tangent bundle $\cH^\star$. 

At that point we consider the left hyperbolic plane
$$
\bH^1_\C=\{\zeta\in\C\;|\;\Re(\zeta)<0\}
$$
with its standard K\"ahler metric tensor
$$
g_h=\frac{|d\zeta|^2}{4\Re^2(\zeta)}.
$$
\begin{defn}
The Kor\'anyi map $\alpha:\fH^\star\to \bH^1_\C=\calL$ is defined by
$$
\alpha(z,t)=-|z|^2+it.
$$
\end{defn}
\begin{prop}
Let $\alpha:\fH^\star\to\calL$ be the Kor\'anyi map as above. The sub-Riemannian metric  $g^\star_{cc}$ of $\fH^\star$  satisfies
$$
g^\star_{cc}=\alpha^*(g_h).
$$
Moreover, $\alpha_*$ is a linear isometry from the horizontal space $\cH^\star=\{\bX,\bY\}$ of $\fH^\star$ endowed with the sub-Riemannian metric $g^\star_{cc}$ to the tangent space ${\rm T}(\calL)$ of $\calL$ endowed with the hyperbolic metric $g_h$.
\end{prop}
\begin{proof}
Let $\zeta=\xi+i\eta$ be the complex coordinate on $\calL$. Then
$$
g_h=ds^2=\frac{d\xi^2+d\eta^2}{4\xi^2},
$$
and the vector fields
$$
\Xi=-2\xi\partial_\xi,\quad {\rm H}=-2\xi\partial_\eta,
$$
form an orthonormal frame for ${\rm T}(\calL)$. The matrix $D\alpha$ of the differential $\alpha_*:{\rm T}(\fH^\star)\to{\rm T}(\calL)$ is
$$
D\alpha=\left(\begin{matrix}
-2x&-2y&0\\
0&0&1\end{matrix}\right).
$$
Straightforward calculations then show that 
$$
\alpha_*(\bX)=\Xi,\quad \alpha_*(\bY)={\rm H},\quad \alpha_*(\bT)=0.
$$
This proves our claim.
\end{proof}
Since the integral curves of the Reeb field $\bT$ are Euclidean circles with centres on the $t$-axis, we also have
\begin{cor}
$\alpha:\fH^\star\to\calL$ defines a circle bundle. Also, $\fH^\star$ is diffeomorphic to $\calL\times S^1$ with the diffeomorphism given by
$$
(z,t)\mapsto(\alpha(z,t),\arg(z)).
$$ 
\end{cor}
Moreover, we have
\begin{cor}\label{cor:ccisom}
The set of hyperbolic Heisenberg isometries ${\rm Isom}(g^\star_{cc})$ is ${\rm SU}(2,1)\times {\rm U}(1)$.
\end{cor}
\begin{proof}
If $F=(f_I,f_3):\fH^\star\to\fH^\star$ is an isometry of $g^\star_{cc}$ then it has to preserve the fibres of $\alpha$. Therefore it defines a map $f:\calL\to\calL$ such that $f\circ\alpha=\alpha\circ F$. But $f$ has to be an isometry, therefore
$$
F(\zeta)=\frac{a\zeta+ib}{ic\zeta+d},\quad \left[\begin{matrix}
a&ib\\
ic&d\end{matrix}\right]\in{\rm SU}(2,1).
$$
From $f\circ\alpha=\alpha\circ F$ we then obtain
$$
|f_I|^2(z,t)=\frac{|z|^2}{|ic(-|z|^2+it)+d},\quad f_3(z,t)=\Im\left(\frac{a(-|z|^2+it)+ib}{ic(-|z|^2+it)+d}\right).
$$
Thus, for some $\theta\in\R$ we have
$$
f_I(z,t)=\frac{ze^{i\theta}}{ic(-|z|^2+it)+d},
$$
and our claim is proved, since the action of ${\rm SU}(1,1)\times{\rm U}(1)$ on $\fH^\star$ is given by self-mappings of the form
\begin{equation}\label{eq:fSU}
F(z,t)=\left(\frac{ze^{i\theta}}{ic(-|z|^2+it)+d},\;\Im\left(\frac{a(-|z|^2+it)+ib}{ic(-|z|^2+it)+d}\right)\right),
\end{equation}
as this follows from (\ref{eq:embedding}).
\end{proof}
\subsection{Horizontal curves, CC-distance}\label{sec:hc}
Let $\gamma:[a,b]\to\C_{*}\times\R$ be a smooth curve. If $\dot\gamma$ is the tangent vector field along $\gamma$ we may write
\begin{eqnarray*}
\dot\gamma(s)&=&\dot x(s)\partial_x+\dot y(s)\partial_y+\dot t(s)\partial_t\\
&=&\frac{\dot x(s)}{|z(s)|^2}(x(s)\bX-y(s)\bT)+\frac{\dot y(s)}{|z(s)|^2}(x(s)\bT+y(s)\bX)+\frac{\dot t(s)}{2|z(s)|^2}(\bT-\bY)\\
&=&\frac{x(s)\dot x(s)+y(s)\dot y(s)}{|z(s)|^2}\bX-\frac{\dot t(s)}{2|z(s)|^2}\bY+\frac{\dot t(s)+2(x(s)\dot y(s)-y(s)\dot x(s))}{2|z(s)|^2}\bT.
\end{eqnarray*}
Therefore
\begin{eqnarray*}
&&
\langle \dot\gamma, \bX\rangle_{\gamma(s)}=\Re\left(\frac{\dot z(s)}{z(s)}\right),\\
&&
\langle \dot\gamma, \bY\rangle_{\gamma(s)}=\Im\left(\frac{\dot z(s)}{z(s)}\right)-\frac{\dot t(s)+2\Im(\overline{z(s)}\dot z(s))}{2|z(s)|^2},\\
&&
\langle \dot\gamma, \bT\rangle_{\gamma(s)}=\frac{\dot t(s)+2\Im(\overline{z(s)}\dot z(s))}{2|z(s)|^2}.
\end{eqnarray*}
\begin{prop}\label{prop-RAlength}
If $\gamma(s)=(z(s),t(s))$, $z(s)=x(s)+iy(s)$, is as above
then $\gamma$ is horizontal if and only if $$\langle\dot\gamma,\bT\rangle_{\gamma(s)}=0,$$ for a.e. $s$. In this case, its horizontal length is
\begin{equation*}
 \ell_{\fH^\star}(\gamma)=\int_a^b\sqrt{\langle\dot\gamma,\bX\rangle^2_{\gamma(s)}+\langle\dot\gamma,\bY\rangle^2_{\gamma(s)}}ds=\int_a^b\frac{|\dot z(s)|}{|z(s)|}ds.
\end{equation*}
\end{prop}

\begin{prop}
If $\gamma$ is a horizontal curve in $\fH^\star$, then
\begin{equation}
\ell_{\fH^\star}(\gamma)=\ell_h(\gamma^*), 
\end{equation}
where $\ell_h(\gamma^*)$ is the hyperbolic length of the curve $\gamma^*=\alpha\circ\gamma$. Here, $\alpha$ is the Kor\'anyi map.
\end{prop}
\begin{proof} 
If $\gamma$ is horizontal then $\dot t(s)=-2\Im({\overline z(s)}\dot z(s))$. For the hyperbolic length $\ell_h(\gamma^*)$ of $\gamma^*$ we have:
\begin{eqnarray*}
 \ell_h(\gamma^*)&=&\int_a^b\frac{|\dot\gamma^*(s)|}{-2\Re(\gamma^*(s))}ds\\
 &=&\int_a^b\frac{|-2\Re(\overline{z(s)}\dot z(s))-2i\Im(\overline{z(s)}\dot z(s))|}{2|z(s)|^2}ds\\
 &=&\int_a^b\frac{|\overline{z(s)}\dot z(s)|}{|z(s)|^2}ds\\
 &=&\ell_{\fH^\star}(\gamma).
\end{eqnarray*}
\end{proof}

The following proposition shows that the horizontal distribution $\cH^\star=\langle \bX,\bY\rangle$ is Ehresmann complete.
\begin{prop}\label{prop-liftcurves}
Suppose that $\gamma^*:[a,b]\to\calL$ is a smooth curve starting from $p^*\in\calL=\gamma^*(a)$. Then for every $p\in\alpha^{-1}(p^*)$ there exists a horizontal $\gamma:[a,b]\to\fH^\star$ starting from $p$ and such that $\alpha\circ\gamma=\gamma^*$ and $\ell_\fH^\star(\gamma)=\ell_h(\gamma^*)$. Here, $\alpha$ is the Kor\'anyi map.
\end{prop}
\begin{proof}
Set $\gamma^*(s)=\zeta(s)$, $\gamma^*(a)=\zeta_0$. Then
$$
\alpha^{-1}(\zeta_0)=\left\{((\Re^{1/2}(-\zeta_0)e^{i\theta},\Im(\zeta_0))\;|\;\theta\in\R\right\}.
$$
Pick a $\theta_0\in\R$ and define $\gamma:[a,b]\to\fH^\star$ from the relation
$$
\gamma(s)=(z(s),t(s))=\left(\Re^{1/2}(-\zeta(s))e^{i\theta(s)},\Im(\zeta(s))\right),
$$
where
$$
\theta(s)-\theta_0=\int_a^s\frac{\Im(\dot\zeta(u))}{2\Re(\zeta(u))}du.
$$
It is clear that $\gamma$ starts from $((\Re^{1/2}(-\zeta_0)e^{i\theta_0},\Im(\zeta_0))$ and we show that $\gamma$ is horizontal:
\begin{eqnarray*}
 \frac{\dot t(s)}{2|z(s)|^2}&=&\frac{\Im(\dot\zeta(s))}{-2\Re(\zeta(s))}\\
 &=&-\dot\theta(s)=-\dot{\arg(z(s))}.
\end{eqnarray*}
If $\gamma^*$ is closed, $\gamma^*(a)=\gamma^*(b)=\zeta_0\in\calL$, then by setting $\zeta=\xi+i\eta$ we have
$$
\theta(b)-\theta(a)=\theta(b)-\theta_0=\int_a^b\frac{\Im(\dot\zeta(s))}{2\Re(\zeta(s))}ds=\int_{\gamma^*}\frac{d\eta}{2\xi}.
$$
Applying Stokes' Theorem we get
$$
\theta(b)=-\iint_{{\rm int}(\gamma^*)}\frac{d\xi\wedge d\eta}{2\xi^2}=-2{\rm Area}_h({\rm int}(\gamma^*)).
$$
Thus
$$
\gamma(a)=(\Re^{1/2}(-\zeta_0)e^{i\theta_0},\Im(\zeta_0)),\quad \gamma(b)=(\Re^{1/2}(-\zeta_0)e^{i(\theta_0-2{\rm Area}_h({\rm int}(\gamma^*))},\Im(\zeta_0)).
$$
The last statement of the proposition is obvious.
\end{proof}
Since $\fH^\star$ is contact, any two points $p$ and $q$ in $\fH^\star$ can be joined with a horizontal geodesic.
\begin{defn}
The {\it Carnot-Carath\'eodory distance} of two arbitrary points $p,q\in\fH^\star$ is defined by
$$
d^\star_{cc}(p,q)=\inf_\gamma\ell_{\fH^\star}(\gamma),
$$ 
where $\gamma$ is horizontal and joins $p$ and $q$.
\end{defn}
According to Corollary \ref{cor:ccisom}, ${\rm Isom}(d^\star_{cc})={\rm SU}(1,1)\times{\rm U}(1)$.

\section{Quasiconformal diffeomorphisms of $\fH^\star$ }\label{sec:qcstar}
\subsection{Smooth contact diffeomorphisms of $\fH^\star$ } Let $F=(f_I,f_3)$ be a smooth (at least of class $\cC^2$)  orientation-preserving diffeomorphism of $\fH^\star$ which is a contact transformation, that is,  $F^*\omega^\star=\lambda^\star\omega^\star$, for some positive function $\lambda^\star$. In other words,
\begin{equation}\label{eq:differentials}
d\arg(f_I)+\frac{df_3}{2|f_I|^2}=\lambda^\star\omega^\star.
\end{equation}
The following is clear.
\begin{prop}
If $F$ is a contact transformation of $\fH^\star$, then $F$ is also a contact transformation with respect to the contact form $\omega$ of $\fH$.
\end{prop}
We shall see below the equivalent equations to (\ref{eq:differentials}). First, we set $f_{II}=-|f_I|^2+if_3$. 
The Jacobian matrix of the differential $F_*$ may be expressed as follows:
\begin{eqnarray*}\label{eq:Jac}
\notag DF&=&\left[\begin{matrix}
\left\langle \phi^\star+i\psi^\star, F_*\bZ\right\rangle&\left\langle  \phi^\star+i\psi^\star, F_*\overline{\bZ}\right\rangle&\left\langle  \phi^\star+i\psi^\star, F_*\bT\right\rangle\\
\\
\left\langle \phi^\star-i\psi^\star, F_*\bZ\right\rangle&\left\langle \phi^\star-i\psi^\star, F_*\overline{\bZ}\right\rangle&\left\langle \phi^\star-i\psi^\star, F_*\bT\right\rangle\\
\\
\left\langle \omega^\star,F_*\bZ\right\rangle&\left\langle \omega^\star,F_*\overline{\bZ}\right\rangle&\left\langle \omega^\star,F_*\bT\right\rangle
      \end{matrix}\right]\\
      &=&
      \left[\begin{matrix}
        -\frac{\bZ f_{II}}{2|f_I|^2}&-\frac{\overline{\bZ} f_{II}}{2|f_I|^2}&-\frac{\bT f_{II}}{2|f_I|^2}\\
        \\
        -\frac{\bZ \overline{f_{II}}}{2|f_I|^2}&-\frac{\overline{\bZ} \overline{f_{II}}}{2|f_I|^2}&-\frac{\bT \overline{f_{II}}}{2|f_I|^2}\\
        \\
        0&0&\lambda^\star
       \end{matrix}\right].
\end{eqnarray*}
We prove the above equality. In the first place 
\begin{equation*}
 \left\langle \phi^\star+i\psi^\star,\;F_*\bZ\right\rangle=\left(F^*\left(\frac{d|z|^2-idt}{2|z|^2}\right)\right)(\bZ)
 =-\frac{df_{II}}{2|f_I|^2}(\bZ)=-\frac{\bZ f_{II}}{2|f_I|^2}.
\end{equation*}
and analogously for the other coefficients of the first two rows. As for the third row,
\begin{eqnarray*}
&&
\left\langle \omega^\star,F_*\bZ\right\rangle=(F^*\omega^\star)(\bZ)=\lambda^\star\omega^\star(\bZ)=0,\\
&&
\left\langle \omega^\star,F_*\overline{\bZ}\right\rangle=(F^*\omega^\star)(\overline{\bZ})=\lambda^\star\omega^\star(\overline{\bZ})=0,\\
&&
\left\langle \omega^\star,F_*\bT\right\rangle=(F^*\omega^\star)(\bT)=\lambda^\star\omega^\star(\bT)=\lambda^\star,
\end{eqnarray*}
since $F$ is contact. These equalities induce the {\it contact conditions for $\cC^2$ diffeomorphisms of $\fH^\star$}:
\begin{eqnarray}
 &&\notag
 \frac{\bZ f_3}{2|f_I|^2}+\bZ\arg(f_I)=0,\\
 &&\label{eq:contact}
 \frac{\overline{\bZ}f_3}{2|f_I|^2}+\overline{\bZ}\arg(f_I)=0,\\
 &&\notag
 \frac{\bT f_3}{2|f_I|^2}+\bT\arg(f_I)=\lambda^\star.
\end{eqnarray}
From the contact conditions we immediately obtain that the Jacobian matrix $D_F$ of the differential $F_*$ may also be written as
\begin{equation}\label{eq:Jac2}
 D_F=\left[\begin{matrix}
       \bZ{\rm Log}(f_I)&\overline{\bZ}{\rm Log}(f_I)&\bT{\rm Log}(f_I)-i\lambda^\star\\
        \\
        \bZ{\rm Log}(\overline{f_I})&\overline{\bZ}{\rm Log}(\overline{f_I})&\bT{\rm Log}(\overline{f_I})+i\lambda^\star\\
        \\
        0&0&\lambda^\star
       \end{matrix}\right].
\end{equation}

\begin{prop}
The Jacobian determinant $J_F$ of $DF$ satisfies $J_F=(\lambda^\star)^2$. Moreover,
\begin{eqnarray*}
 \lambda^\star&=&|\bZ{\rm Log}(f_I)|^2-|\overline{\bZ}{\rm Log}(f_I)|^2\\
 &=&2\Im\left(\bZ\log(|f_I|)\cdot\overline{\bZ}\arg(f_I)\right)\\
 \label{eq:lafter}&=&-\frac{\Im(\bZ\log(|f_I|)\cdot\overline{\bZ}f_3)}{|f_I|^2}.
\end{eqnarray*}
\end{prop}
\begin{proof}
 We prove the first equality; the second equality is induced after straightforward calculations and the third equality follows from the contact conditions. We have:
 \begin{eqnarray*}
|\bZ{\rm Log}(f_I)|^2-|\overline{\bZ}{\rm Log}(f_I)|^2&=& \det\left(\begin{matrix}
\left\langle \phi^\star+i\psi^\star,F_*\bZ\right\rangle&\left\langle \phi^\star+i\psi^\star,F_*\overline{\bZ}\right\rangle\\
\\
\left\langle \phi^\star-i\psi^\star,F_*\bZ\right\rangle&\left\langle \phi^\star-i\psi^\star,F_*\overline{\bZ}\right\rangle\\
      \end{matrix}\right)\\
      &=&\left\langle \phi^\star+i\psi^\star,F_*\bZ\right\rangle\cdot\left\langle \phi^\star-i\psi^\star,F_*\overline{\bZ}\right\rangle\\
      &&-\left\langle \phi^\star+i\psi^\star,F_*\overline{\bZ}\right\rangle\cdot\left\langle \phi^\star-i\psi^\star,F_*\bZ\right\rangle\\
      &=&-2i\;\phi^\star\wedge\psi^\star\left(F_*\bZ,F_*\overline{\bZ}\right)\\
      &=&F^*\left(-2i\;\phi^\star\wedge\psi^\star\right)(\bZ,\overline{\bZ})\\
      &=&F^*(-id\omega^\star)(\bZ,\overline{\bZ})\\
      &=&-i(d\lambda^\star\wedge\omega^\star+\lambda^\star d\omega^\star)(\bZ,\overline{\bZ})\\
      &=&-i\lambda^\star d\omega^\star(\bZ,\overline{\bZ})=\lambda^\star.
      \end{eqnarray*}
\end{proof}

Denote by $|\cdot|$ the sub-Riemannian norm on $\fH^\star$. Let $F=(f_I,f_3)$ a smooth, orientation-preserving contact transformation, $F^*\omega^\star=\lambda^\star\omega^\star$, and let also $p\in\fH^\star$. Horizontal vectors $V_p\in\cH^\star_p$ are of the form $a\bX_p+b\bY_p$ and $|V_p|=\sqrt{a^2+b^2}$. Since $F$ is contact, for every such $V_p$ we always have $F_{*,p}V_p\in\cH_{F(p)}$. We consider
\begin{equation*}
\lambda_1(p)=\sup_{|V_p|=1}|F_{*,p}V_p|,\quad \lambda_2(p)=\inf_{|V_p|=1}|F_{*,p}V_p|.
\end{equation*}
We have
$$
\lambda_1(p)=|\bZ{\rm Log}(f_I)(p)|+|\overline{\bZ}{\rm Log}(f_I)(p)|
$$
and
$$
\lambda_2(p)=|\bZ{\rm Log}(f_I)(p)|-|\overline{\bZ}{\rm Log}(f_I)(p)|.
$$
Let
$$
D_0(p)=\left[\begin{matrix}
       \bZ{\rm Log}(f_I)&\overline{\bZ}{\rm Log}(f_I)\\
        \\
        \bZ{\rm Log}(\overline{f_I})&\overline{\bZ}{\rm Log}(\overline{f_I})
       \end{matrix}\right]_p,
$$
and consider the matrix $D_0(p)^*D_0(p)$, where $D_0(p)^*$ is the complex transpose of $D_0(p)$. Note that $\lambda_1^2(p),\lambda_2^2(p)$ are the eigenvalues of $D_0(p)^*D_0(p)$ and we have also proved that $J_0(p)=\det(J_0(p))$ $=\lambda^\star(p)$. The following is now clear:

\begin{prop}
For an orientation-preserving, smooth contact transformation $F$ of $\fH^\star$ we have
$$
\lambda_1\lambda_2=\lambda^\star.
$$
\end{prop} 

\begin{defn}\label{defn:mdf}
Let $F=(f_I,f_3)$ a smooth, orientation-preserving contact transformation of $\fH^\star$. The {\it maximal distortion function} $K_F$ is defined for each $p\in\fH^\star$ by
$$
K_F(p)=\frac{\lambda_1(p)}{\lambda_2(p)}=\frac{|\bZ{\rm Log}(f_I)(p)|+|\overline{\bZ}{\rm Log}(f_I)(p)|}{|\bZ{\rm Log}(f_I)(p)|-|\overline{\bZ}{\rm Log}(f_I)(p)|}=\frac{|\bZ f_I(p)|+|\overline{\bZ} f_I (p)|}{|\bZ f_I (p)|-|\overline{\bZ} f_I (p)|}.
$$
\end{defn}

\subsection{Quasiconformal contact diffeomorphisms of $\fH^\star$ }
\begin{defn}
Suppose that $F=(f_I,f_3)$ is a smooth, orientation-preserving contact transformation of $\fH^\star$. If its maximal distortion function satisfies
$$
K_F(p)\le K, \quad p\in\fH^\star,
$$
for some $K$, then $F$ shall be called $K$-quasiconformal.
\end{defn}
The {Beltrami coefficient} $\mu_F:\fH^\star\to\C$ of the $K$-quasiconformal mapping $F$ is given by
$$
\mu_F(p)=\frac{\overline{\bZ}f_I}{\bZ f_I}(p),\quad p\in\fH^\star.
$$
$K$-quasiconformality of $F$ is equivalent to
$$
\|\mu_F\|={\rm esssup}_{p\in\fH^\star}\{|\mu_F(p)|\}\le k=\frac{K-1}{K+1}\in[0,1).
$$
Note that due to contact conditions we also have
\begin{eqnarray*}
&&
\frac{\overline{\bZ}f_I}{\bZ f_I}=\frac{\overline{\bZ}{\rm Log}f_I}{\bZ{\rm Log}f_I}=\frac{\overline{\bZ}f_{II}}{\bZ f_{II}}=\mu_F.
\end{eqnarray*}
The following holds.
\begin{prop}
A contact $K$-quasiconformal transformation of $\fH^\star$ is also a contact $K$-\\quasiconformal transformation of an open and connected subset of the Heisenberg group $\fH$.
\end{prop}
\begin{proof}
 Since $\bZ=zZ$, we have  from Definition \ref{defn:mdf} that the maximal distortion function of $F$ as a contact transformation of $\fH^\star$ coincides with the maximal distortion function of $F$ as a contact transformation of a domain of $\fH$. Note also that if $\mu_F$ (resp. $\tilde{\mu}_F$) is the Beltrami coefficient of $F$ as a contact quasiconformal mapping of $\fH^\star$ (resp. $\fH$), then
 $$
 \mu_F(z,t)=\frac{z}{\overline{z}}\cdot\tilde{\mu}_F(z,t),\quad (z,t)\in\fH^\star.
 $$ 
 \end{proof}
\begin{prop}
The group ${\rm SU}(1,1)\times {\rm U}(1)$  is the group of 1-quasiconformal (conformal) mappings of $\fH^\star$.
\end{prop}
\begin{proof}
Recall from (\ref{eq:fSU}) that
that elements of ${\rm SU}(1,1)\times {\rm U}(1)$ may be represented by self-mappings of $\fH^\star$ of the form
$$
F(z,t)=\left(\frac{ze^{i\theta}}{ic(-|z|^2+it)+d},\;\Im\left(\frac{a(-|z|^2+it)+ib}{ic(-|z|^2+it)+d}\right)\right),
$$
where $ad+bc=1$. It is straightforward to show that $\overline{\bZ}f_I=0$, so we only have to show that $F$ is contact. To do show, we take under account the $KAN$ decomposition of ${\rm SU}(1,1)$:
$$
{\rm SU}(1,1)={\rm SO}(2)\cdot(\R_{>0},\cdot)\cdot(\R,+),
$$
and we examine each component and ${\rm U}(1)$ separately.

a) In the ${\rm U}(1)$ case ($a=d=1$, $b=c=0$), we simply have
$$
F(z,t)=\left(ze^{i\theta},\;t\right).
$$
This is the left-translation $L_{(e^{i\theta},0)}$.

b) In the $(\R,+)$ case ($a=d=1$, $c=0$, $\theta=0$), we obtain
$$
F(z,t)=\left(z, \;t+b\right).
$$
This is the left-translation $L_{(1,b)}$.

c) In the $(\R_{>0},\cdot)$ case $ad=1$, $a>0$, $b=c=\theta=0$, we obtain
$$
F(z,t)=\left(az, \;a^2t\right).
$$
This is the left-translation $L_{(a,0)}$.

d) In the ${\rm SO}(2)$ case ($a=d=\cos\phi$, $b=c=\sin\phi$, $\theta=0$), we have
$$
F(z,t)=\left(\frac{z}{i\sin\phi(-|z|^2+it)+\cos\phi},\;\Im\left(\frac{\cos\phi(-|z|^2+it)+i\sin\phi}{i\sin\phi(-|z|^2+it)+\cos\phi}\right)\right).
$$
Tedious but straightforward calculations show that $F^*(\omega^\star)=\omega^\star$. Note that in all cases $\lambda^\star=1$.

Now, a contact 1-quasiconformal mapping $F$ of $\fH^\star$ is a contact 1-quasiconformal mapping of $\fH$. According to Theorem 8 in \cite{KR1}, $F$ has to be an element of ${\rm PU}(2,1)$. But only ${\rm SU}(1,1)\times{\rm U}(1)$ $\subset{\rm PU}(2,1)$ preserves $\fH^\star$.
\end{proof}

\section{The Lifting Theorem}\label{sec:lift}

\subsection{Contact diffeomorphisms with constant Jacobian}\label{sec:cJ}
Smooth, orientation-preserving contact diffeomorphism with constant Jacobian determinant are of our particular interest. We start with the following. 

\begin{prop}\label{prop:constJ}
Suppose that $F=(f_I,f_3)$ is an orientation-preserving, $\cC^2$ contact diffeomorphism of $\fH^\star$. Then its Jacobian $\lambda^\star=const.$ if and only if the function $f_{II}:\fH^\star\to\calL$,  $f_{II}=$ $-|f_I|^2+if_3$ depends only on $\zeta=-|z|^2+it$. Moreover, in that case we have
 $$
 \arg(f_I)=\arg z +\varphi(\zeta),
 $$
 for some $\cC^2$ function $\varphi:\fH^\star\to\R$.
\end{prop}

For the proof we shall need the next lemma.

\begin{lem}\label{lem-T}
 Let $h:\fH^\star\to\C$ be a $\cC^1$ function. Then if $h=h(\zeta)$, $\zeta=-|z|^2+it$, we have
 \begin{enumerate}
  \item $\bZ h=2\Re(\zeta)h_\zeta$, $\overline{\bZ} h=2\Re(\zeta)h_{\overline \zeta}$ and $\bT h=0$.
  \item $h=h(\zeta)$ if and only if
 $
 \bT h=0.
 $
 \end{enumerate}
\end{lem}
\begin{proof}
The first statement follows by applying the chain rule. As for the second,
 the vector field $\bT$ defines the partial differential equation
 $$
 x\frac{\partial h}{\partial y}-y\frac{\partial h}{\partial x}=0.
 $$
 The solutions of this equation are all of the form $H(|z|^2,t)=0$.
\end{proof}

\medskip

\noindent{\it Proof of Proposition \ref{prop:constJ}}

Suppose first that  $F=(f_I,f_3)$ is an orientation-preserving, $\cC^2$ contact diffeomorphism of $\fH^\star$ with constant $\lambda^\star$, then the mapping
 $$
 \left(\frac{f_I}{\sqrt{\lambda^\star}},\frac{f_3}{\lambda^\star}\right),
 $$
 is also contact and has Jacobian determinant equal to one. Thus we may always normalise so that $\lambda^\star\equiv 1$.
 So suppose first that $\lambda^\star=1$ and take differentials at both sides of the relation (\ref{eq:differentials}) to obtain
 $$
 \frac{d|f_I|^2\wedge df_3}{2|f_I|^4}=d\omega^\star.
 $$
 This is written also as
 $$
 d\log(|f_I|)\wedge df_3=|f_I|^2 d\omega^\star
 $$
 and by taking the differentials at both sides we have
 $$
 d|f_I|^2\wedge d\omega^\star=0.
 $$
 Using
 $$
 d|f_I|^2=\bZ(|f_I|^2)(\phi^\star+i\psi^\star)+\overline{\bZ}(|f_I|^2)(\phi^\star-i\psi^\star)+\bT(|f_I|^2)\omega^\star
 $$
 and
 $$
 d\omega^\star=i\;(\phi^\star+i\psi^\star)\wedge(\phi^\star-i\psi^\star),
 $$
 we get
 $$
 0=i\bT(|f_I|^2)\cdot dm.
 $$
 Therefore, by Lemma \ref{lem-T}  $|f_I|$ depends only on $\zeta$. 
 
 On the other hand, we have
 $$
 0=-id\omega^\star(\bZ,\bT)=-i(F^* d\omega^\star)(\bZ,\bT)=id\omega^\star(F_*\bZ,F_*\bT)=(\phi^\star+i\psi^\star)\wedge(\phi^\star-i\psi^\star)(F_*\bZ,F_*\bT).
 $$
 Thus
 \begin{eqnarray*}
  0&=&(\phi^\star+i\psi^\star)\wedge(\phi^\star-i\psi^\star)(F_*\bZ,F_*\bT)\\
  &=&\langle\phi^\star+i\psi^\star, F_*\bZ\rangle\cdot\langle\phi^\star-i\psi^\star, F_*\bT\rangle-
  \langle\phi^\star+i\psi^\star, F_*\bT\rangle\cdot\langle\phi^\star-i\psi^\star, F_*\bZ\rangle\\
  &=&\frac{1}{4|f_I|^2}\left(\bZ f_{II}\bT\overline{f_{II}}-\bZ\overline{f_{II}}\bT f_{II}\right).
 \end{eqnarray*}
But
$$
\bT f_{II}=i\bT f_3,\quad \bT\overline{f_{II}}=-i\bT f_3.
$$
Thus
$$
\frac{-i\bT f_3}{4|f_I|^2}(\bZ f_{II}+\bZ\overline{f_{II}})=0.
$$
Now $\bZ f_{II}+\bZ\overline{f_{II}}\neq 0$; if that was the case in some point, then $\lambda^\star$ would vanish at this point, a contradiction. Therefore $\bT f_3=0$ and our first claim is proved.

We now prove that if $f_{II}$ depends only on $\zeta$, that is, $\bT f_{II}=0$, then $\lambda^\star=const.$ For this, we consider the differential of $\lambda^\star$:
$$
d\lambda^\star=\bZ\lambda^\star(\phi^\star+i\psi^\star)+\overline{\bZ}\lambda^\star(\phi^\star-i\psi^\star)+\bT\lambda^\star\omega^\star
$$
and we will show that it is equal to zero. Since
$$
F^*d\omega^\star=d(F^*\omega^\star)=d(\lambda^\star\omega^\star)=d\lambda^\star\wedge\omega^\star+\lambda^\star d\omega^\star,
$$
we have
$$
F^*d\omega^\star=\bZ\lambda^\star(\phi^\star+i\psi^\star)\wedge\omega^\star+\overline{\bZ}\lambda^\star(\phi^\star-i\psi^\star)\wedge\omega^\star+i\lambda^\star(\phi^\star+i\psi^\star)\wedge(\phi^\star-i\psi^\star).
$$
Now, in the first place
$$
F^*d\omega^\star(\bZ,\bT)=\bZ\lambda^\star(\phi^\star+i\psi^\star)\wedge\omega^\star(\bZ,\bT)=\bZ\lambda^\star.
$$
On the other hand,
\begin{eqnarray*}
 F^*d\omega^\star(\bZ,\bT)&=&d\omega^\star(F_*\bZ,F_*\bT)\\
 &=&i(\phi^\star+i\psi^\star)\wedge(\phi^\star-i\psi^\star)(F_*\bZ,F_*\bT)\\ 
 &=&\frac{i}{4|f_I|^2}\left(\bZ f_{II}\bT\overline{f_{II}}-\bT f_{II}\bZ\overline{f_{II}}\right)\\
 &=&0.
\end{eqnarray*}
Thus $\bZ\lambda^\star=0$ and similarly $\overline{\bZ}\lambda^\star=0$. Finally, we have that
$$
\lambda^\star=\frac{|\bZ f_{II}|^2-|\overline{\bZ} f_{II}|^2}{4|f_I|^4}.
$$
It suffices to prove that $\bT\bZ f_{II}=0$. Indeed,
$$
\bT\bZ f_{II}=-\bZ\bT f_{II}=0
$$
and the proof is concluded.
\qed

\subsection{The Lifting Theorem}

Smooth contact maps of $\fH^\star$ with constant Jacobian admit the following geometric interpretation. 

\begin{defn}
 A bijection $F:\fH^\star\to\fH^\star$ is called {\it circles-preserving} if for every $\zeta\in\calL$ there exists an $\eta\in\calL$ such that
 $$
 F(\alpha^{-1}(\zeta))=\alpha^{-1}(\eta).
 $$
\end{defn}
That is, a circles-preserving mapping $F=(f_I,f_3)$ preserves the fibres of the Kor\'anyi map $\alpha$. Such a mapping defines a bijection $f:\calL\to\calL$ by the rule
$$
f\circ\alpha=\alpha\circ F=f_{II}.
$$
We now immediately have the following:
\begin{cor}
An orientation-preserving, $\cC^2$ contact diffeomorphism $F=(f_I,f_3)$ of $\fH^\star$ has constant $\lambda^\star$ if and only if $F$ is circles-preserving.
\end{cor}


We are now set to prove our main theorem. We start with a lemma.

\begin{lem}\label{lem:psi}
Let $ f:\calL\to\calL$ be a smooth symplectic diffeomorphism. Then there exists a function $\psi:\calL\to\R$ such that
\begin{equation}\label{eq:psi}
 \psi_\zeta=\frac{i}{4\Re(\zeta)}+\frac{(\Im( f))_\zeta}{2\Re( f)}.
\end{equation}
\end{lem} 
\begin{proof}
Since $f$ is symplectic we have
\begin{equation}\label{eq:kift2}
 \Re^2(f)(\zeta)=\Re^2(\zeta) J_{f}(\zeta),
\end{equation}
where $J_f$ is the Jacobian determinant of $f$.
We have
\begin{eqnarray*}
J_{f}&=&|f_\zeta|^2-|f_{\overline{\zeta}}|^2\\
&=&|(\Re(f))_\zeta+i(\Im(f))_\zeta|^2-|(\Re(f))_{\overline \zeta}+i(\Im(f))_{\overline \zeta}|^2\\
&=&|(\Re(f))_\zeta+i(\Im(f))_\zeta|^2-|(\Re(f))_{ \zeta}-i(\Im(f))_{\zeta}|^2\\
&=&-4\Im\left((\Re(f))_\zeta(\Im(f))_{\overline \zeta}\right)\\
&=&2i\left((\Re(f))_\zeta(\Im(f))_{\overline \zeta}-(\Re(f))_{\overline \zeta}(\Im(f))_{\zeta}\right).
\end{eqnarray*}
Therefore we write (\ref{eq:kift2}) again as
\begin{equation}\label{eq:kift3}
\Re^2(f)(\zeta)=2i\Re^2(\zeta)\left((\Re(f))_\zeta(\Im(f))_{\overline \zeta}-(\Re(f))_{\overline \zeta}(\Im(f))_{\zeta}\right).
\end{equation}
We next set
$$
g(\zeta)=\frac{i}{4\Re(\zeta)}+\frac{(\Im(f))_\zeta}{2\Re(f)}.
$$
Then
$$
g_{\overline \zeta}=-\frac{i}{8\Re^2(\zeta)}+\frac{(\Im(f))_{\zeta\overline\zeta}\cdot\Re(f)-(\Re(f))_{\overline \zeta}(\Im(f))_{\zeta}}{8\Re^2(f)}
$$
and
$$
{\overline g}_{\zeta}=\frac{i}{8\Re^2(\zeta)}+\frac{(\Im(f))_{\zeta\overline\zeta}\cdot\Re(f)-(\Re(f))_{\zeta}(\Im(f))_{\overline \zeta}}{8\Re^2(f)}.
$$
Using (\ref{eq:kift3}) we find that
$$
{\overline g}_{\zeta}-g_{\overline{\zeta}}=0.
$$
We next consider the real 1-form
$$
\beta=g(\zeta)d\zeta+\overline{g(\zeta)}d\overline{\zeta}.
$$
By the previous equation we have that $\beta$ is closed and hence exact by Poincar\'e's Lemma. It follows that
$$
\beta=d\psi
$$
for some real function $\psi$ and the proof is concluded.
\end{proof}
\begin{rem}
Write $\zeta=\xi+i\eta$ and $f=u+iv$. Then Equation (\ref{eq:psi}) is written equivalently as the system of p.d.e's:
\begin{equation}\label{eq:psisys}
\psi_\xi=\frac{v_\xi}{2u},\quad \psi_\eta=\frac{v_\eta}{2u}-\frac{1}{2\xi}.
\end{equation}
\end{rem}
We now state and prove our main theorem.
\begin{thm}{\bf (Lifting Theorem)}\label{thm:lift}
Let $f:\calL\to\calL$ be a mapping which is symplectic with respect to the K\"ahler symplectic form of $\calL$. Let also $\psi$ be as in Lemma \ref{lem:psi}.
 Then the map $F:\fH^\star\to\fH^\star$, $F=(f_I,f_3)$, where
 $$
 f_I(z,t)=zJ_{ f}^{1/4}(-|z|^2+it)e^{i\psi(-|z|^2+it)},\quad f_3(z,t)=\Im( f(-|z|^2+it)),
 $$
 is an orientation-preserving, circles-preserving contact transformation of $\fH^\star$.
 If moreover $f$ is quasiconformal with Beltrami coefficient $\mu_f$, then $F$ is quasiconformal and
 $$
 \mu_F(z,t)=\mu_f(-|z|^2+it).
 $$ 
$f$ is symplectic w.r.t. the K\"ahler form of $\calL$.
\end{thm}
\begin{proof}
 In the first place, $f_{II}(z,t)=-|z|^2J_{\tilde f}^{1/2}(\zeta)+i\Im(\tilde f)(\zeta)=f(\zeta)$. Moreover,
 $$
 \lambda^\star=\frac{J_{f_{II}}}{4|f_I|^4}=\frac{4\Re^2(\zeta)J_{ f}}{4\Re^2(\zeta)J_{ f}}=1.
 $$
 Now,
 \begin{eqnarray*}
  \frac{\bZ f_3}{2|f_I|^2}+\bZ\arg(f_I)&=&\frac{2\Re(\zeta)(\Im( f(\zeta)))_\zeta}{-2\Re(\zeta)J_{ f}^{1/2}}+\frac{1}{2i}+2\Re(\zeta)\psi_\zeta\\
  &=&-\Re(\zeta)\frac{(\Im( f(\zeta)))_\zeta}{\Re( f)}+\frac{1}{2i}+2\Re(\zeta)\psi_\zeta\\
  &=&0.
 \end{eqnarray*}
 Here, we have used Lemma \ref{lem:psi}.
Similarly,
$$
 \frac{\overline{\bZ} f_3}{2|f_I|^2}+\overline{\bZ}\arg(f_I)=0
$$
and finally,
$$
\frac{\bT f_3}{2|f_I|^2}+\bT\arg(f_I)=\bT\arg(z)=1.
$$
Now, since $f_{II}(z,t)=f(\zeta)$, we have from the chain rule that
\begin{eqnarray*}
&&
\overline\bZ f_{II}=f_\zeta\overline{\bZ}\zeta+f_{\overline{\zeta}}\overline{\bZ\zeta}=-2|z|^2f_{\overline{\zeta}},\\
&&
\bZ f_{II}=f_\zeta\bZ\zeta+f_{\overline{\zeta}}\bZ\overline{\zeta}=-2|z|^2f_\zeta.
\end{eqnarray*}
Therefore,
$$
\frac{\overline{\bZ}f_{II}}{\bZ f_{II}}=\frac{f_{\overline{\zeta}}}{f_\zeta}=\mu_f(\zeta).
$$
On the other hand, recall that due to contact conditions we also have
\begin{eqnarray*}
&&
\mu_F(z,t)=\frac{\overline{\bZ}f_I}{\bZ f_I}=\frac{\overline{\bZ}{\rm Log}f_I}{\bZ{\rm Log}f_I}=\frac{\overline{\bZ}f_{II}}{\bZ f_{II}}=\mu_f(\zeta).
\end{eqnarray*}
This completes the proof.
\end{proof}
\subsection{Examples of lifted symplectic maps}\label{sec:examples}
\subsubsection{Lifting ${\rm SU}(1,1)$} We will use the Lifting Theorem \ref{thm:lift} to lift isometries of $\calL$ to conformal maps of $\fH^*$.  
Let $f:\calL\to\calL$ be an element of ${\rm SU}(1,1)$; for $z\in\calL$, 
$$
f(\zeta)=\frac{a\zeta+ib}{ic\zeta+d},\quad ad+bc=1.
$$
$f$ is symplectic and $\mu_f(z)\equiv 0$ in $\calL$. We have
$$
f_\zeta=f'(\zeta)=\frac{1}{(ic\zeta+d)^2},\quad J_f(\zeta)=\frac{1}{|ic\zeta+d|^4}.
$$
On the other hand,
$$
\Re(f)(\zeta)=\frac{\Re(\zeta)}{|ic\zeta+d|^2},\quad(\Im(f))_\zeta=\frac{f_\zeta}{2i}.
$$
Therefore Equation (\ref{eq:psi}) of Lemma \ref{lem:psi} is written here as
$$
\psi_\zeta=-\frac{c}{2(ic\zeta+d)}.
$$
Let $\zeta=\xi+i\eta$. We examine first the case where $c=0$ ($a=1/d)$; then from Equations \ref{eq:psisys} we obtain
$$
\psi_\xi=\psi_\eta=0.
$$
Thus $\psi(\zeta)\equiv \alpha$, a real constant. For the lifted map $F=(f_I,f_3)$ of Theorem \ref{thm:lift} we then have
$$
f_I(z,t)=|a|z e^{i\alpha},\quad f_3(z,t)=a^2t+ab.
$$
By setting $w=|a|e^{i\alpha}$ and $s=b$ we see that $F$ is the left-translation $L_{(w,s)}$.

When $c\neq 0$, we have the equations
$$
\psi_\xi=\frac{c(c\eta-d)}{c^2\xi^2+(c\eta-d)^2},\quad \psi_\eta=-\frac{c^2\xi}{c^2\xi^2+(c\eta-d)^2}.
$$
By integrating the left equation we find
$$
\psi(\xi,\eta)=\arctan\left(\frac{\xi}{\eta-d/c}\right)+\beta(\eta),
$$ 
where $\beta(\eta)$ is a real function. By differentiating w.r.t. $\eta$ we obtain
$$
\psi_\eta=-\frac{c^2\xi}{c^2\xi^2+(c\eta-d)^2}+\beta'(\eta).
$$
Therefore $\beta(\eta)=\beta$ is a real constant and
$$
\psi(\xi,\eta)=\arctan\left(\frac{\xi}{\eta-d/c}\right)+\beta.
$$
In this manner we obtain an $F=(F_I,f_3)$ of the form (\ref{eq:fSU}).
\subsubsection{Lifting twist  maps}
Let $g(\theta)$ be a smooth (at least $\cC^1$) function, $\theta=\arg(\zeta)$. We consider the map $f:\calL\to\calL$ given by
$$
f(\zeta)=\zeta e^{g(\theta)}.
$$
We have
\begin{eqnarray*}
&&
f_{\overline{\zeta}}=-\frac{g'f}{2i\overline{\zeta}},\quad f_\zeta=\frac{(2i+g')f}{2i\zeta},\\
&&
J_f=e^{2g},\quad \mu_f=\frac{g'}{g'+2i}\frac{\zeta}{\overline{\zeta}}.
\end{eqnarray*}
The map $f$ is symplectic:
$$
\Re^2(f(\zeta))=\Re^2(\zeta)e^{2g}.
$$
Equation (\ref{eq:psi}) is written as
$$
\psi_\zeta=0,
$$
which gives $\psi=c$, a real constant. Therefore the lifted map $F=(f_I,f_3)$ is given by
$$
f_I(z,t)=ze^{g(\arg(-|z|^2+it))/2+ic},\quad f_3(z,t)=te^{g(\arg(-|z|^2+it))}.
$$
In the case where $g(\theta)=k\theta$, $k\in\R_*$ we have
$$
|\mu_f(\zeta)|=\frac{|k|}{\sqrt{k^2+4}},
$$
and $f$ is extremal. We conclude that in this case we also have that the lifted map $F$ where
$$
F(z,t)=\left(ze^{\frac{k}{2}\arg(-|z|^2+it)}e^{ic},\;te^{k\;\arg(-|z|^2+it)}\right),
$$
is extremal. Note that $F$ maps each Heisenberg cylinder $t=\alpha|z|^2$, $\alpha\in\R$, to itself.

\subsubsection{Lifting spiral-stretch maps}
The standard stretch map $\zeta\mapsto\zeta|\zeta|^k$, $k>0$, is not symplectic, so it cannot be lifted to a map of $\fH^\star$. Instead we consider a smooth function $g(\theta)$, $\theta=\arg(\zeta)$, and the map
$$
f(\zeta)=\zeta|\zeta|^ke^{ig(\theta)},
$$
where we also allow the case $k=0$.
We have
\begin{eqnarray*}
&&
f_{\overline{\zeta}}=\frac{(k-g')f}{2\overline{\zeta}},\quad
f_\zeta=\frac{(2+k+g')f}{2\zeta},\\
&&
J_f=(1+k)(1+g')|\zeta|^{2k},\quad \mu_f=\frac{k-g'}{2+k+g'}\frac{\zeta}{\overline{\zeta}}.
\end{eqnarray*}
Straightforward calculations show that $f$ is symplectic if and only if
$$
\frac{1}{1+k}\frac{\cos^2(g+\theta)}{\cos^2\theta}=\frac{d(g+\theta)}{d\theta},
$$
that is, 
$$
g(\theta)=\arctan\left(\frac{\tan\theta}{k+1}+k'\right)-\theta,\quad k'\in\R,
$$
and 
$$
f(\zeta)=|\zeta|^{k+1}e^{i\arctan\left(\frac{\tan\theta}{1+k}+k'\right)}.
$$
Note that
$$
g'(\theta)=\frac{(k+1)(1+\tan^2\theta)}{(k+1)^2+(\tan\theta+k'(k+1))^2}-1.
$$
Setting $\xi=r\cos\theta$ and $\eta=r\sin\theta$, the system (\ref{eq:psisys}) may be written equivalently as
\begin{equation}\label{eq:polar}
\psi_r=\frac{v_r}{2u}-\frac{\tan\theta}{2r},\quad \psi_\theta=\frac{v_\theta}{2u}-\frac{1}{2}.
\end{equation}
Now,
$$
u(r,\theta)=r^{k+1}\cos\left(\arctan\left(\frac{\tan\theta}{k+1}+k'\right)\right),\quad v(r,\theta)=r^{k+1}\sin\left(\arctan\left(\frac{\tan\theta}{k+1}+k'\right)\right).
$$
So, the system (\ref{eq:polar}) reads as
\begin{eqnarray*}
&&
\psi_r=\frac{k'(k+1)}{2r},\\
&&
\psi_\theta=\frac{g'(\theta)}{2}.
\end{eqnarray*}
By integrating the above system we find
$$
\psi(r,\theta)=\frac{k'(k+1)}{2}\log r+\frac{g(\theta)}{2}+c,
$$
where $c$ is a real constant.
From Theorem \ref{thm:lift} we then obtain the quasiconformal (Heisenberg spiral) map constructed in \cite{P},  if $k=0$ and the Heisenberg stretch map constructed in \cite{BFP} (modulo a rotation).

\end{document}